\documentclass[12pt, a4paper]{amsart}

\usepackage{fullpage}
\usepackage{amsmath,amssymb}
\usepackage[colorlinks=true]{hyperref}
\usepackage{xypic}
\usepackage{xfrac}
\usepackage[usenames]{color}
\usepackage{euscript}

\newcommand{\Z}{\mathbb{Z}} 
\newcommand{\Q}{\mathbb{Q}}
\newcommand{\R}{\mathbb{R}} 
 
\newcommand{\F}{\mathbb{F}}

\usepackage{todonotes}

\makeatletter
\providecommand\@dotsep{5}
\def\listtodoname{List of Todos}
\def\listoftodos{\@starttoc{tdo}\listtodoname}
\makeatother

\theoremstyle{plain}
\newtheorem{theorem}{Theorem}[section]
\newtheorem{lemma}[theorem]{Lemma}
\newtheorem{corollary}[theorem]{Corollary}
\newtheorem{prop}[theorem]{Proposition}
\newtheorem{prop-def}[theorem]{Proposition / Definition}
\newtheorem{sec:hyp}[theorem]{Hypothesis}
\newtheorem{algorithm}[theorem]{Algorithm}

\newtheorem*{problem*}{Problem}
\newtheorem*{theorem*}{Theorem}

\theoremstyle{remark}

\newtheorem{remark}[theorem]{Remark}

\newtheorem{example}[theorem]{Example}

\newtheorem*{note*}{Note}
\newtheorem*{remark*}{Remark}
\newtheorem*{example*}{Example}

\theoremstyle{definition}
\newtheorem*{definition*}{Definition}
\newtheorem*{hypothesis*}{Hypothesis}
\newtheorem*{hypotheses*}{Hypotheses}
\newtheorem*{assumptions*}{Assumptions}

\newcommand{\Gal}{\mathrm{Gal}}

\newcommand{\disc}{\mathrm{Disc}}

\newcommand{\Nm}{\mathrm{N}}

\newcommand{\Cl}{\mathrm{Cl}}

\newcommand{\tr}{\mathrm{tr}}
\newcommand{\Aut}{\mathrm{Aut}}
\newcommand{\End}{\mathrm{End}}
\newcommand{\Mat}{\mathrm{Mat}}
\newcommand{\Hom}{\mathrm{Hom}}

\newcommand{\GL}{\mathrm{GL}}

\newcommand{\SL}{\mathrm{SL}}

\newcommand{\E}{\mathrm{E}}

\newcommand{\calA}{\mathcal{A}}

\newcommand{\calG}{\mathcal{G}}
\newcommand{\calO}{\mathcal{O}}
\newcommand{\calM}{\mathcal{M}}

\newcommand{\fra}{\mathfrak{a}}
\newcommand{\frb}{\mathfrak{b}}
\newcommand{\frc}{\mathfrak{c}}

\newcommand{\frf}{\mathfrak{f}}
\newcommand{\frg}{\mathfrak{g}}
\newcommand{\frh}{\mathfrak{h}}
\newcommand{\frp}{\mathfrak{p}}
\newcommand{\frP}{\mathfrak{P}}
\newcommand{\frq}{\mathfrak{q}}

\newcommand{\lra}{\longrightarrow}
\newcommand{\nr}{\mathrm{nr}}

\newcommand{\OC}{\calO_C}

\newcommand{\OK}{{\calO_K}}

\newcommand{\sseq}{\subseteq}

\newcommand{\im}{\mathrm{im}}

\newcommand{\mpar}[1]{}

\newcommand{\psimilar}{\ensuremath{\mathsf{IsSimilar}}}

\newcommand{\ppip}{\ensuremath{\mathsf{IsPrincipal}}}
\newcommand{\psplitting}{\ensuremath{\mathsf{SplittingMatrixAlgebra}}}
\newcommand{\pprimitive}{\ensuremath{\mathsf{Primitive}}}
\newcommand{\unit}{\ensuremath{\mathsf{UnitGroup}}}
\newcommand{\pwedderburn}{\ensuremath{\mathsf{Wedderburn}}}
\newcommand{\pfactor}{\ensuremath{\mathsf{Factor}}}
\newcommand{\pdlog}{\ensuremath{\mathsf{DLog}}}
\newcommand{\pisomorphic}{\ensuremath{\mathsf{IsIsomorphic}}}
\newcommand{\punit}{\ensuremath{\mathsf{UnitGroup}}}
\DeclareMathOperator{\G}{G}
\DeclareMathOperator{\jac}{J}

\title{Computation of lattice isomorphisms and the integral matrix similarity problem}

\author{Werner Bley}
\address{
Ludwig-Maximilians-Universität München\\
Theresienstr.\ 39\\
D-80333 M\"unchen\\
Germany}
\email{bley@math.lmu.de}
\urladdr{https://www.mathematik.uni-muenchen.de/$\sim$bley/}

\author{Tommy Hofmann}
\address{
Naturwissenschaftlich-Technische Fakult\"at\\
Universit\"at Siegen\\
Walter-Flex-Straße 3\\
57068 Siegen\\
Germany}
\email{tommy.hofmann@uni-siegen.de}

\author{Henri Johnston}
\address{
Department of Mathematics\\
University of Exeter\\
Exeter\\
EX4 4QF\\
United Kingdom
}
\email{H.Johnston@exeter.ac.uk}
\urladdr{http://emps.exeter.ac.uk/mathematics/staff/hj241}

\subjclass[2000]{11R33, 11Y40, 16Z05, 20G30}
\keywords{}
\date{31st August 2022}

\begin{document}

\begin{abstract}
Let $K$ be a number field, let $A$ be a finite-dimensional $K$-algebra, 
let $\jac(A)$ denote the Jacobson radical of $A$, and let $\Lambda$ be
an $\mathcal{O}_{K}$-order in $A$. 
Suppose that each simple component of the semisimple $K$-algebra 
$A/{\jac(A)}$ is isomorphic to a matrix ring over a field.
Under this hypothesis on $A$, we give an algorithm that given two $\Lambda$-lattices 
$X$ and $Y$, determines whether $X$ and $Y$ are isomorphic, and if so, computes an 
explicit isomorphism $X \rightarrow Y$. 
This algorithm reduces the problem to standard problems in computational algebra and algorithmic algebraic number theory in polynomial time. 
As an application, we give an algorithm for the following long-standing problem: 
given a number field $K$, a positive integer $n$ and two matrices
$A,B \in \Mat_{n}(\mathcal{O}_{K})$, determine whether $A$ and $B$ are similar over $\mathcal{O}_{K}$,
and if so, return a matrix $C \in \GL_{n}(\mathcal{O}_{K})$ such that $B= CAC^{-1}$.
We give explicit examples that show that the implementation of the latter algorithm
for $\mathcal{O}_{K}=\Z$ vastly outperforms implementations of all previous algorithms,
as predicted by our complexity analysis.
\end{abstract}

\maketitle

\section{Introduction}

Let $K$ be a number field with ring of integers $\mathcal{O}_{K}$. 
Let $A$ be a finite-dimensional $K$-algebra and 
let $\Lambda$ be an $\mathcal{O}_{K}$-order in $A$.
A \textit{$\Lambda$-lattice} is a (left) $\Lambda$-module that is finitely generated and torsion-free over 
$\mathcal{O}_{K}$.
We will consider the following problem.

\begin{problem*}[$\pisomorphic$] 
Given two $\Lambda$-lattices $X$ and $Y$, decide whether $X$ and $Y$ are isomorphic, and if so, 
return an isomorphism $X \to Y$.
\end{problem*}

A $\Lambda$-lattice contained in $A$ is said to be \textit{full} if it contains a $K$-basis of $A$.
We will show that $\pisomorphic$ is polynomial-time reducible (see \S \ref{sec:prelim-complexity})
to the following problem.

\begin{problem*}[$\ppip$] 
Given a full $\Lambda$-lattice $X$ in $A$, decide whether there exists $\alpha \in X$ such that
$X = \Lambda \alpha$, and if so, return such an element $\alpha$.
\end{problem*}

Let $\jac(A)$ denote the Jacobson radical of $A$.
Note that the quotient algebra $\overline{A} := A/{\jac(A)}$ is semisimple.  
Let $h : A \rightarrow \overline{A}$ denote the canonical projection map
and let $\overline{\Lambda} = h(\Lambda)$. 
We will show that the problem $\ppip$ for a full $\Lambda$-lattice $X$ in $A$
is polynomial-time reducible to the problem $\ppip$ for the full $\overline{\Lambda}$-lattice 
$\overline{X}$ in $\overline{A}$, where $\overline{X}=h(X)$.

Let
\[
A/{\jac(A)} \simeq \bigoplus_{i=1}^{r} A_{i}
\]
be the Wedderburn decomposition. 
Each simple component $A_{i}$ is isomorphic to a matrix ring $\Mat_{n_i}(D_i)$ where
$D_i$ is a skew field extension of $K$. Let $K_{i}$ denote the centre of $D_{i}$.
In order to make progress on the above problems, we impose the following hypothesis.

\begin{itemize}
\item[(H)] Each component $A_{i}$ of the Wedderburn decomposition $A/{\jac(A)} \simeq \bigoplus_{i=1}^{r} A_{i}$ is isomorphic to a matrix ring over a field.
\end{itemize}
In the above notation, this is equivalent to the assertion that $D_{i}=K_{i}$ for each $i$.

Under hypothesis (H), we give algorithms that solve both $\pisomorphic$ and $\ppip$.
Moreover, we give the first complexity analysis of these problems and thus
prove the following result.
For precise definitions and statements, we refer the reader to \S \ref{sec:prelim-complexity} and 
\S \ref{sec:pip}.

\begin{theorem*}
The problem $\pisomorphic$ for lattices over orders in algebras 
satisfying hypothesis
\textup{(H)} reduces in probabilistic polynomial time to
\begin{enumerate}
\item \pwedderburn, the problem of computing explicitly the Wedderburn decomposition,
\item \pfactor, the problem of factoring integers,
\item \ppip{} in the special case of rings of integers of number fields,
\item \punit, the computation of unit groups for rings of integers of number fields,
\item \pprimitive, the computation of primitive elements in finite fields, and
\item \pdlog, the computation of discrete logarithms in finite fields.
\end{enumerate}
\end{theorem*}

A number of articles have considered $\pisomorphic$, $\ppip$ or closely related problems in special cases.
In particular, \cite{bley-endres} applies in the case that $A$ is commutative and semisimple;
\cite{BW09} applies to group rings $\mathcal{O}_{K}[G]$ where $G$ is a finite group, but only decides whether two lattices are both locally free and stably isomorphic; and 
\cite{MR2467859,Kirschmer2010,MR3240815} apply to maximal or Eichler orders in quaternion algebras.
The series of articles \cite{Bley1997, BJ08, BJ11, MR4136552} consider progressively more
general situations, culminating in a solution to $\pisomorphic$ when $A$ is semisimple, but they all involve a very expensive enumeration step, which in many cases renders the algorithm impractical.
We refer the reader to the introduction of \cite{MR4136552} for a more detailed overview. 
By contrast, Algorithm~\ref{alg:main} replaces this enumeration step by a new method combining results of
\cite{bley-boltje, BW09} with an idea of Husert \cite{Hus17}.

The original motivation for the study of these problems comes from the Galois module structure of rings of
integers. Let $L/K$ be a finite Galois extension of number fields and let $G=\Gal(L/K)$.
An interesting but difficult problem is to determine whether 
$\mathcal{O}_{L}$ is free over its so-called associated order
$\mathcal{A}_{L/K} = \{ \alpha \in K[G] \mid \alpha\mathcal{O}_{L} \subseteq \mathcal{O}_{L} \}$, 
and if so, to determine an explicit generator. We refer the reader to \S \ref{sec:GMS} and to the 
introduction of \cite{MR4136552} for a more detailed overview of this question and related problems.

The main application of $\ppip$ in the present article is to the following problem.

\begin{problem*}[$\psimilar$] 
Given a number field $K$ with ring of integers $\mathcal{O}=\mathcal{O}_{K}$,
an integer $n \in \Z_{>0}$ and two matrices 
$A,B \in \Mat_{n}(\mathcal{O})$, determine whether $A$ and $B$ are similar over $\mathcal{O}$,
and if so, return a conjugating matrix $C \in \GL_{n}(\mathcal{O})$ such that $B= CAC^{-1}$.
\end{problem*}

As a special case, this problem includes the so-called conjugacy problem for $\GL_{n}(\mathcal{O})$.
A number of authors have considered the problem $\psimilar$ (or special cases), including 
Latimer--MacDuffee~\cite{MR1503108},
Sarkisyan \cite{zbMATH03665057}, Grunewald \cite{MR579942},
Husert \cite{Hus17} and Marseglia \cite{MR4111932}.
Eick--O'Brien and the second named author of the present article gave the first
practical algorithm that solves this problem in full generality \cite{MR4048718}.
We refer the reader to \S \ref{subsec:alg-similarity} and \S \ref{subec:similarity-implementation}
for a more detailed discussion of these results.

In \S \ref{sec:conjugacy}, we give an efficient algorithm that solves $\psimilar$ 
in full generality and a complexity analysis showing that it 
is polynomial-time reducible to standard problems
in algorithmic algebraic number theory, 
including the principal ideal problem in certain rings of integers
and the computation of their unit groups
(see Algorithm \ref{alg:similar} and Theorem \ref{thm:isconjugated}).
As a corollary we obtain the following result (see Corollary~\ref{cor:isconjugated1} and Remark~\ref{remark:subcomp}).

\begin{theorem*}
The problem
  $\psimilar$ reduces in probabilistic subexponential time to the problems \ppip{} and \punit{} for rings of integers of number fields.
\end{theorem*}

We first adapt ideas of Faddeev \cite{MR0194432} to recast $\psimilar$ in terms of 
lattices over orders in a certain $K$-algebra satisfying hypothesis (H). 
We then show how to explicitly compute the 
Jacobson radical of this $K$-algebra as well as the Wedderburn decomposition of the semisimple quotient
from the rational canonical forms of the input matrices. 
Thus we show that $\psimilar$ is reducible to $\ppip$. 
In particular, Algorithm \ref{alg:similar} avoids any expensive enumeration step.
For a detailed comparison with other algorithms and implementations, 
including explicit examples and timings, 
we refer the reader to \S \ref{subec:similarity-implementation}.
As these comparisons and our complexity analysis suggest, 
the implementation of Algorithm \ref{alg:similar} in 
the computer algebra package~\textsc{Hecke}~\cite{Fieker2017}
vastly outperforms implementations of other algorithms.

\subsection*{Acknowledgements}
The authors wish to thank Nigel Byott, Fabio Ferri, Claus Fieker, and J\"urgen Kl\"uners 
for useful conversations, and are grateful for numerous helpful comments and corrections
from Nigel Byott, Gunter Malle, Stefano Marseglia, and an anonymous referee.
The second named author was supported by Project II.2 of SFB-TRR 195 `Symbolic Tools in
Mathematics and their Application'
of the German Research Foundation (DFG).

\section{Preliminaries on lattices and orders}\label{sec:prelims-on-lattices-orders}

For further background on lattices and orders, we refer the reader to \cite[\S 4, \S 8]{Reiner2003}.
Henceforth all rings considered will be associative and unital. 

Let $R$ be an noetherian integral domain with field of fractions $K$.
To avoid trivialities, we assume that $R \neq K$.
An \textit{$R$-lattice} is a finitely generated torsion-free module over $R$.
Since $R$ is noetherian, any $R$-submodule of an $R$-lattice is again an $R$-lattice.
For any finite dimensional $K$-vector space $V$, an \textit{$R$-lattice in $V$} is a finitely generated $R$-submodule $M$ in $V$.
We define a $K$-vector subspace of $V$ by 
\[
K M := \{ \alpha_{1} m_{1} + \alpha_{2} m_{2} + \cdots + \alpha_{r} m_{r} \mid r \in \Z_{\geq 0}, \alpha_i \in K, m_{i} \in M \}
\]
and say that $M$ is a \textit{full} $R$-lattice in $V$ if $K M=V$. 
We may identify $KM$ with $K \otimes_{R} M$.

Now let $A$ be a finite-dimensional $K$-algebra.
Then $A$ is both left and right artinian and noetherian.
An \textit{$R$-order} in $A$ is a subring $\Lambda$ of $A$ 
(so in particular has the same unit element as $A$) such that $\Lambda$
is a full $R$-lattice in $A$.
Note that $\Lambda$ is both left and right noetherian, since $\Lambda$ is finitely generated over $R$. 
A left \textit{$\Lambda$-lattice $X$} is a left $\Lambda$-module that is also an $R$-lattice; in this case, $KX$ may be viewed as a left $A$-module.

Henceforth all modules (resp.\ lattices) will be assumed to be left modules (resp.\ lattices) unless otherwise stated.
Two $\Lambda$-lattices are said to be \textit{isomorphic} if they are isomorphic as $\Lambda$-modules.
The following two lemmas generalise \cite[Lemma 2.1]{MR4136552}.

\begin{lemma}\label{lem:lattice-extension}
Let $S$ be a noetherian integral domain such that $R \subseteq S \subsetneq K$. 
Let $\Gamma$ be an $S$-order in $A$.
Let $V$ be a finitely generated $A$-module.
For any $R$-lattice $M$ in $V$, the set
\[
\Gamma M := \{ \gamma_{1} m_{1} + \gamma_{2} m_{2} + \cdots + \gamma_{r} m_{r} \mid r \in \Z_{\geq 0}, m_{i} \in M, \gamma_{i} \in \Gamma \}
\]
is a $\Gamma$-lattice in $V$ containing $M$.
\end{lemma} 

\begin{proof}
That $M \subseteq \Gamma M$ is clear.
Note that $K$ is the field of fractions of both $R$ and $S$. 
Write $M = \langle v_{1}, \ldots, v_{l} \rangle_{R}$ and $\Gamma = \langle w_{1}, \ldots, w_{m} \rangle_{S}$.
An easy calculation shows that
\[
\Gamma M = \langle w_{i} v_{j} \mid 1 \leq i \leq m, 1 \leq j \leq l \rangle_{S}
\]
and hence $\Gamma M$ is an $S$-lattice in $V$.
Moreover, it is straightforward to see that $\Gamma M$ is also a $\Gamma$-module and therefore is a $\Gamma$-lattice in $V$.
\end{proof}

\begin{lemma}\label{lem:ext}
Let $S$ be a noetherian integral domain such that $R \subseteq S \subsetneq K$.
Let $\Lambda$ be an $R$-order in $A$, let $\Gamma$ be an $S$-order in $A$ and suppose that 
$\Lambda \subseteq \Gamma$.
Let $f \colon X \to Y$ be a homomorphism of $\Lambda$-lattices.
Then the following hold.
\begin{enumerate}
\item There exists a unique homomorphism of $A$-modules $f^{A} \colon KX \to KY$ extending $f$.
\item There exists a unique homomorphism of $\Gamma$-lattices $f^{\Gamma}\colon \Gamma X \to \Gamma Y$ extending $f$.
\item If $f$ is injective (resp.\ surjective), then $f^{A}$ and $f^{\Gamma}$ are injective (resp.\ surjective).
\end{enumerate}
\end{lemma}

\begin{proof}
This is straightforward. The key points are to
(a) extend $f$ to $KX$ using $K$-linearity; (b) restrict $f^{A}$ to $\Gamma X$; (c) (injectivity) check that $\ker(f)$ is a full $R$-lattice in $\ker(f^{A})$; and (c) (surjectivity) use the definitions of $K Y$ and
$\Gamma Y$.
\end{proof}

We will often use the following result without explicit mention.

\begin{lemma}\label{lem:cyclic-iff-free}
Let $\Lambda$ be an $R$-order in $A$ and let $X$ be a $\Lambda$-lattice such that 
$\dim_{K} KX = \dim_{K} A$. Let $\alpha \in X$. 
Then $X=\Lambda \alpha$ if and only if $\alpha$ is a free generator of $X$ over $\Lambda$.
\end{lemma}

\begin{proof}
Suppose $X = \Lambda \alpha$. 
Then the map $f : \Lambda \rightarrow X$ given by $f(\lambda) = \lambda\alpha$
is a surjective homomorphism of $\Lambda$-lattices. 
By Lemma \ref{lem:ext} $f$ extends uniquely to a surjective map $f^{A} : A \rightarrow KX$.
The hypotheses imply that $f^{A}$ is injective, thus $f$ is an isomorphism and so
$\alpha$ is a free generator of $X$ over $\Lambda$.
The converse is trivial. 
\end{proof}

\section{Reduction steps for the lattice isomorphism problem}\label{sec:reduce-lattice-isom}

Let $R$ be a noetherian integral domain with field of fractions $K$ and assume that $R \neq K$.
Let $\Lambda$ be an $R$-order in a finite-dimensional $K$-algebra $A$. 

\subsection{Reduction to the free rank $1$ case via homomorphism groups}\label{subsec:reduce-free-rank1}
Let $X$ and $Y$ be $\Lambda$-lattices. Let $V=KX$ and $W=KY$, which we regard as $A$-modules. 
We have 
\[
\Hom_\Lambda(X, Y) = \{ f|_{X} \mid f \in \Hom_{A}(V, W) \text{ such that $f(X) \subseteq Y$} \}, 
\]
where $f|_{X}$ denotes the restriction of $f$ to a map $f \colon X \to Y$. 
This follows from the fact that every element in $\Hom_{\Lambda}(X, Y)$ extends uniquely to an element in $\Hom_{A}(V, W)$ (see Lemma~\ref{lem:ext}).
Since a map $f \in \Hom_{A}(V,W)$ is also $R$-linear, we have $f(X) \subseteq Y$ if and only if
$f \in \Hom_{R}(X,Y)$. Therefore
\[
\Hom_{\Lambda}(X, Y) =\Hom_{A}(V, W) \cap \Hom_{R}(X, Y).
\]
Since $X$ and $Y$ are finitely generated over $R$, so is $\Hom_{R}(X,Y)$.
Therefore $\Hom_{\Lambda}(X,Y)$ is a full $R$-lattice in $\Hom_{A}(V,W)$.
Similarly, $\End_{\Lambda}(Y)$ is a full $R$-lattice in $\End_{A}(W)$. 

In fact, $\End_{\Lambda}(Y)$ is an $R$-order in $\End_{A}(W)$
and $\Hom_{\Lambda}(X,Y)$ is a (left) $\End_{\Lambda}(Y)$-lattice in
$\Hom_{A}(V,W)$ via post-composition. 
The following result underpins the main results of the present article; it is a straightforward generalisation of \cite[Proposition 3.7]{MR4136552}.

\begin{prop}\label{prop:hom-free-over-end}
Two $\Lambda$-lattices $X$ and $Y$ are isomorphic if and only if
\renewcommand{\labelenumi}{(\alph{enumi})}
\begin{enumerate}
\item the $\End_{\Lambda}(Y)$-lattice $\Hom_{\Lambda}(X, Y)$ is free of rank $1$, and
\item every (any) free generator of $\Hom_{\Lambda}(X, Y)$ over $\End_{\Lambda}(Y)$ is an isomorphism.
\end{enumerate}
\end{prop}

\begin{proof}
If (a) and (b) hold then it is clear that $X$ and $Y$ are isomorphic. 
Suppose conversely that $X$ and $Y$ are isomorphic.
Fix an isomorphism $\varphi \in \Hom_{\Lambda}(X, Y)$.
Then for any $g \in \Hom_{\Lambda}(X, Y)$, we have $h_{g} := g \circ \varphi^{-1} \in \End_{\Lambda}(Y)$ and so $g = h_{g} \circ \varphi$.
Hence $\varphi$ is a generator of $\Hom_{\Lambda}(X, Y)$ over $\End_{\Lambda}(Y)$
and by Lemma \ref{lem:cyclic-iff-free} it is in fact a free generator.
Thus (a) holds. Now let $f$ be any free generator of $\Hom_{\Lambda}(X, Y)$ over $\End_{\Lambda}(Y)$. 
Then there exists $\theta \in \Aut_{\Lambda}(Y)=\End_{\Lambda}(Y)^{\times}$ such that
$f = \theta \circ \varphi$ and hence $f$ is an isomorphism.
Thus (b) holds.
\end{proof}

We now state and prove a closely related `folklore' result that appears to be well known, but whose proof is difficult to locate in the literature.
We include this result for completeness and it will not be applied in the present article.
For any full $R$-lattice $M$ in $A$ we define
$\mathcal{O}_{r}(M) = \{ \mu \in A \mid M\mu \subseteq M \}$.
This is an $R$-order in $A$ and is called the \emph{right order} of $M$ in $A$ (see \cite[\S 8]{Reiner2003}).
The following result may be viewed as a corollary of Proposition \ref{prop:hom-free-over-end},
but it is easier to give a direct proof.

\begin{prop}\label{prop:lattices-in-A-isomorphic}
Let $X$ and $Y$ be full $\Lambda$-lattices in $A$. 
Then $C := \{ \lambda \in A \mid X\lambda \subseteq Y \}$ is a full $\mathcal{O}_{r}(X)$-lattice in $A$.
Moreover, $X$ and $Y$ are isomorphic if and only if 
\begin{enumerate}
\item there exists $\alpha \in A^{\times}$ such that $C=\mathcal{O}_{r}(X) \alpha$, and
\item we have $Y=XC$.
\end{enumerate}
Furthermore, when this is the case, $Y=X\alpha$.
\end{prop}

\begin{proof}
Set $\mathcal{O}:=\mathcal{O}_{r}(X)$.
Clearly, $C$ is both an $R$-module and an $\mathcal{O}$-module.
Since $X$ and $Y$ are both full $R$-lattices in $A$ there exist nonzero $r,s \in R$
such that $Ys \subseteq X$ and $Xr \subseteq Y$ (see \cite[\S 4]{Reiner2003}). 
Thus $\mathcal{O}r \subseteq C \subseteq \mathcal{O}s^{-1}$,
where $\mathcal{O}r$ and $\mathcal{O}s^{-1}$ are both full $R$-lattices in $A$.
Hence $C$ is a full $R$-lattice and therefore a full $\mathcal{O}$-lattice in $A$.

Suppose (a) and (b) hold. Then $Y=XC=X(\mathcal{O}\alpha) = (X\mathcal{O})\alpha = X \alpha$.
Hence $X$ and $Y$ are isomorphic since $\alpha \in A^{\times}$. 
Suppose conversely that $f: X \rightarrow Y$ is a $\Lambda$-isomorphism.
Then by Lemma \ref{lem:ext} $f$ extends uniquely to an $A$-isomorphism $f^{A} : A \rightarrow A$,
and hence is given by right multiplication by an element $\alpha \in A^{\times}$.
Thus $Y = X\alpha$. 
Moreover, 
$C = \{ \lambda \in A \mid X\lambda \subseteq X\alpha \} = \mathcal{O} \alpha$
and $Y = X\alpha =  (X\mathcal{O})\alpha = X(\mathcal{O}\alpha) = XC$.
\end{proof}

\subsection{An alternative approach via localisation}
We give an alternative version of Proposition \ref{prop:hom-free-over-end} that uses localisation.
This will be useful later for understanding the relation of some of our results to other results in the literature.
For a nonzero prime ideal $\mathfrak{p}$ of $R$, we
let $R_{\mathfrak{p}}$ denote the localisation (not completion) of $R$ at $\mathfrak{p}$.
We define the \textit{localisation} $M_{\mathfrak{p}}$ of $M$ at $\mathfrak{p}$ to be
$R_{\mathfrak{p}}M$ and note that this is an $R_{\mathfrak{p}}$-lattice in $KM$.
The localisation $\Lambda_{\mathfrak{p}}$ is an 
$R_{\mathfrak{p}}$-order in $A$ and localising a $\Lambda$-lattice $X$ 
at $\mathfrak{p}$ yields
a $\Lambda_{\mathfrak{p}}$-lattice $X_{\mathfrak{p}}$.  
Two $\Lambda$-lattices $X$ and $Y$ are said to be \textit{locally isomorphic} 
if the $\Lambda_{\mathfrak{p}}$-lattices $X_{\mathfrak{p}}$ and $Y_{\mathfrak{p}}$
are isomorphic for all maximal ideals $\mathfrak{p}$ of $R$. 

\begin{prop}\label{prop:hom-end-locally-iso}
Two $\Lambda$-lattices $X$ and $Y$ are isomorphic if and only if
\renewcommand{\labelenumi}{(\alph{enumi})}
\begin{enumerate}
\item $X$ and $Y$ are locally isomorphic, and
\item the $\End_{\Lambda}(Y)$-lattice $\Hom_{\Lambda}(X, Y)$ is free of rank $1$.
\end{enumerate} 
Furthermore, when this is the case, every free generator of 
 $\Hom_{\Lambda}(X, Y)$ over $\End_{\Lambda}(Y)$ is an isomorphism.
\end{prop}

\begin{proof}
If $X$ and $Y$ are isomorphic then (a) clearly holds and (b) holds by Proposition \ref{prop:hom-free-over-end}. Suppose conversely that (a) and (b) hold. 
Let $f$ be a free generator of $\Hom_{\Lambda}(X, Y)$ over $\End_{\Lambda}(Y)$.
Let $\mathfrak{p}$ be a maximal ideal of $R$.
Then there exists an isomorphism 
$g_{\mathfrak{p}} \in \Hom_{\Lambda_{\mathfrak{p}}}(X_{\mathfrak{p}}, Y_{\mathfrak{p}})$.
Moreover, $f$ extends to a free generator $f_{\mathfrak{p}}$ of
$\Hom_{\Lambda_{\mathfrak{p}}}(X_{\mathfrak{p}}, Y_{\mathfrak{p}})$ over 
$\End_{\Lambda_{\mathfrak{p}}}(Y_{\mathfrak{p}})$,
and so there exists $h_{\mathfrak{p}} \in \End_{\Lambda_{\mathfrak{p}}}(Y_{\mathfrak{p}})$
such that $g_{\mathfrak{p}} = h_{\mathfrak{p}} \circ f_{\mathfrak{p}}$.
Note that $h_{\mathfrak{p}}$ is surjective and thus is in fact an automorphism of $Y_{\mathfrak{p}}$
by \cite[(5.8)]{curtisandreiner_vol1}.
Therefore $f_{\mathfrak{p}}$ is an isomorphism. 
Since this is true for all choices of $\mathfrak{p}$, 
we have that $f$ itself is an isomorphism by \cite[(4.2)(ii)]{curtisandreiner_vol1}.
\end{proof}

\subsection{Reduction to the case of lattices in semisimple algebras}\label{subsec:red-to-ss}

Let $\jac(A)$ denote the Jacobson radical of $A$ and note that
$\overline{A} := A/{\jac(A)}$ is a semisimple $K$-algebra by \cite[(5.19)]{curtisandreiner_vol1}.
Let $h : A \rightarrow \overline{A}$ denote the canonical projection map. 
For an element $a \in A$ write $\overline{a}$ for $h(a)$ and for a subset $S \subseteq A$ write
$\overline{S}$ for $h(S)$.
Then $\overline{\Lambda}$ is an $R$-order in $\overline{A}$.
The following result may be viewed as a variant of \cite[Theorem 3]{MR0194432}.

\begin{theorem}\label{thm:lift-gen-from-semisimple-case}
Let $X$ be a full $\Lambda$-lattice in $A$. 	
Then $\overline{X}$ is a full $\overline{\Lambda}$-lattice in $\overline{A}$.
Moreover, the following statements hold for $\alpha \in X$.
\begin{enumerate}
\item If $X = \Lambda \alpha$ then $\overline{X} = \overline{\Lambda} \overline{\alpha}$.
\item If  $\overline{X} = \overline{\Lambda} \overline{\alpha}$ 
then either $X = \Lambda \alpha$ or $X \ne \Lambda \beta$ for all $\beta \in X$.
\end{enumerate}
\end{theorem}

\begin{proof}
The first claim and part (a) are both clear.
Suppose that $\overline{X} = \overline{\Lambda} \overline{\alpha}$ and that there exists
$\beta \in X$ such that $X=\Lambda \beta$.
Since $\alpha \in X=\Lambda \beta$ there exists $\varepsilon \in \Lambda$
such that $\alpha = \varepsilon \beta$. 
Hence $\overline{\alpha} = \overline{\varepsilon}\overline{\beta}$ and
since each of $\overline{\alpha}$ and $\overline{\beta}$ is a free  generator
of $\overline{X}$ over $\overline{\Lambda}$,
we must have $\overline{\varepsilon} \in \overline{\Lambda}^{\times}$.
Let $\eta \in \Lambda$ such that $\overline{\eta}=\overline{\varepsilon}^{-1}$.
Then $\varepsilon \eta = 1 + \rho$ where $\rho \in \jac(A)$. 
Moreover, $\rho = \varepsilon \eta - 1 \in  \Lambda$.
Since $A$ is artinian, $\jac(A)$ is a nilpotent ideal (see \cite[(5.15)]{curtisandreiner_vol1})
and so $\rho$ is a nilpotent element. 
Therefore
\[
\varepsilon^{-1} 
= \eta(1 + \rho)^{-1}
= \eta(1 - \rho + \rho^{2} - \rho^{3} + \dotsb) \in \Lambda,
\]
where the alternating sum is finite.
Hence $\varepsilon \in \Lambda^{\times}$ and so
$\Lambda \alpha = \Lambda \varepsilon \beta = \Lambda \beta = X$.
\end{proof}

\section{A necessary and sufficient condition for freeness}

Let $K$ be a number field with ring of integers $\mathcal{O}=\mathcal{O}_{K}$, 
and let $A$ be a finite-dimensional semisimple $K$-algebra.
Let $\Lambda$ be an $\mathcal{O}$-order in $A$.
By \cite[(10.4)]{Reiner2003} there exists a (not necessarily unique) maximal $\mathcal{O}$-order $\mathcal{M}$ in $A$ containing $\Lambda$.

Lemmas \ref{lem:intersection-addition-ideals} and \ref{lem:lattice-intersection}, as well as part of Proposition \ref{prop:equiv-conditions}, are based on \cite[\S 1.6]{Hus17}.

\begin{lemma}\label{lem:intersection-addition-ideals}
Let $\mathfrak{c}, \mathfrak{d}, \mathfrak{f}$ be left ideals of $\Lambda$ such that $\mathfrak{c} \subseteq \mathfrak{d}$. Then
$
\mathfrak{d} \cap (\mathfrak{c} + \mathfrak{f}) = \mathfrak{c} + (\mathfrak{d} \cap \mathfrak{f})
$.
\end{lemma}

\begin{proof}
We follow the proof of \cite[Lemma 1.37]{Hus17}.
If $c \in \mathfrak{c}$ and $f \in \mathfrak{f}$ with $c+f \in \mathfrak{d}$ then
$f \in \mathfrak{d}$ since $c \in \mathfrak{d}$.
Hence $c+f \in \mathfrak{c} + (\mathfrak{d} \cap \mathfrak{f})$. 
Therefore $\mathfrak{d} \cap (\mathfrak{c} + \mathfrak{f}) \subseteq \mathfrak{c} + (\mathfrak{d} \cap \mathfrak{f})$. For the reverse inclusion, note that both $\mathfrak{c}$ and 
$\mathfrak{d} \cap \mathfrak{f}$ are contained in
$\mathfrak{d} \cap (\mathfrak{c} + \mathfrak{f})$, and thus the same is true for their sum. 
\end{proof}

An ideal of a ring will be said to be \emph{proper} if the containment is strict.
Henceforth let $\mathfrak{f}$ be any proper full two-sided ideal of $\mathcal{M}$ that is contained in $\Lambda$.
For $\eta \in \mathcal{M}$ we write $\overline{\eta}$ for its image in $\mathcal{M} / \mathfrak{f}$.

\begin{lemma}\label{lem:lattice-intersection}
Let $X$ be a left ideal of $\Lambda$. 
If $X + \mathfrak{f} = \Lambda$ and $\beta \in \mathcal{M}$ such that 
$\mathcal{M}X = \mathcal{M}\beta$ then $\overline{\beta} \in (\mathcal{M}/\mathfrak{f})^{\times}$ and 
$\mathcal{M}X \cap \Lambda = X$.
\end{lemma}

\begin{proof}
We adapt the proof of \cite[Lemma 1.38]{Hus17}.
Clearly $\mathfrak{f}X \subseteq \mathcal{M}X \cap \mathfrak{f}$.
We now show the reverse inclusion.
Let $\gamma \in \mathcal{M} X \cap \mathfrak{f}$.
Then we write $\gamma = \lambda \beta$ with $\lambda \in \mathcal{M}$
and we have
\begin{align*}
X + \mathfrak{f} = \Lambda 
&\implies \mathcal{M}(X + \mathfrak{f}) = \mathcal{M} \\
& \implies \mathcal{M}\beta + \mathfrak{f} = \mathcal{M} \\
& \implies \overline{\beta} \in (\mathcal{M} / \mathfrak{f})^{\times} \\
&\implies \overline{\lambda} = \overline{\gamma} \overline{\beta}^{-1} = \overline{0} (\overline{\beta})^{-1} =  \overline{0} \textrm{ in } \mathcal{M}/\mathfrak{f}\\
& \implies \lambda \in \mathfrak{f} \\
&\implies \gamma=\lambda\beta \in \mathfrak{f}\beta = \mathfrak{f}\mathcal{M}\beta = \mathfrak{f}\mathcal{M} X = \mathfrak{f} X.
\end{align*}
Therefore $\mathfrak{f}X = \mathcal{M}X \cap \mathfrak{f}$.
Moreover, we have
\[
\mathcal{M} X \cap \Lambda = \mathcal{M} X \cap (X+\mathfrak{f}) = X + (\mathcal{M} X \cap \mathfrak{f}) = X + \mathfrak{f} X = X
\]
where the second equality holds by Lemma \ref{lem:intersection-addition-ideals}. 
\end{proof}

Define
\[
\pi : (\mathcal{M}/\mathfrak{f})^{\times}
\longrightarrow
{(\mathcal{M} / \mathfrak{f} )^{\times}} / {(\Lambda / \mathfrak{f} )^{\times}}
\]
to be the map induced by the canonical projection, where the codomain is the 
collection of left cosets of
$ (\Lambda / \mathfrak{f} )^{\times}$ in $(\mathcal{M} / \mathfrak{f} )^{\times}$.
Note that $(\Lambda/\mathfrak{f})^{\times}$ is a subgroup of $(\mathcal{M}/\mathfrak{f})^{\times}$, but is not necessarily a normal subgroup, and so $\pi$ is only a map of sets in general.

Part of the following result is a variant of \cite[Theorem 1.39]{Hus17}.

\begin{prop}\label{prop:equiv-conditions}
Let $X$ be a left ideal of $\Lambda$. 
Suppose that $X + \mathfrak{f} = \Lambda$ and that there exists $\beta \in \mathcal{M}$ such that 
$\mathcal{M}X = \mathcal{M}\beta$. Let $u \in \mathcal{M}^{\times}$ and let $\alpha = u \beta$.
Then $\overline{\alpha}, \overline{\beta}, \overline{u} \in (\mathcal{M}/\mathfrak{f})^{\times}$ and 
the following are equivalent:
\begin{enumerate}
\item $X = \Lambda\alpha$,
\item $\Lambda\alpha + \mathfrak{f} = \Lambda$,
\item $\overline{\alpha} \in (\Lambda/\mathfrak{f})^{\times}$,
\item $\pi(\overline{\beta}) = \pi(\overline{u^{-1}})$,
\item $\alpha \in X$ and $X$ is locally free over $\Lambda$.
\end{enumerate}
\end{prop}

\begin{proof}
Lemma \ref{lem:lattice-intersection} and the definitions of $u$ and $\alpha$ imply that 
$\overline{\alpha}, \overline{\beta}, \overline{u} \in (\mathcal{M}/\mathfrak{f})^{\times}$.
It is clear that (b) $\Leftrightarrow$ (c).
Since $\beta = u^{-1} \alpha$, that (c) $\Leftrightarrow$ (d) follows from the definition of $\pi$.
Since $X + \mathfrak{f} = \Lambda$, we also have (a) $\Rightarrow$ (b).
Assume (b) holds.
By two applications of Lemma~\ref{lem:lattice-intersection} we have
\[
X = \mathcal{M}X \cap \Lambda = \mathcal{M}\beta \cap \Lambda
= \mathcal{M}\alpha \cap \Lambda = 
\mathcal{M}(\Lambda\alpha) \cap \Lambda = \Lambda\alpha,
\]
where the first equality uses the hypothesis that  $X + \mathfrak{f} = \Lambda$ and the last equality uses the assumption that (b) holds; thus (a) holds. Therefore (a) $\Leftrightarrow$ (b).
Finally, a special case of \cite[Proposition 2.1]{BJ08} shows that (a) $\Leftrightarrow$ (e).
\end{proof}

Much of the following notation is adopted from \cite{bley-boltje} and \cite{MR4136552}.
Denote the centre of a ring $R$ by $Z(R)$. 
Let $C=Z(A)$ and let $\mathcal{O}_{C}$ be the integral closure of $\mathcal{O}$ in $C$.
Let $\mathfrak{g} = \mathfrak{f} \cap C$ and note that this is a proper full ideal of $\mathcal{O}_{C}$.
Let $e_{1}, \ldots, e_{r}$ be the primitive idempotents
of $C$ and set $A_{i} = Ae_{i}$. Then
\begin{equation}\label{eq:weddecomp}
A = A_{1} \oplus \cdots \oplus A_{r}
\end{equation}
is a decomposition of $A$ into indecomposable two-sided ideals (see \cite[(3.22)]{curtisandreiner_vol1}).
Each $A_{i}$ is a simple $K$-algebra with identity element $e_{i}$.
The centres $K_{i} := Z(A_{i})$ are finite field extensions of $K$ via $K \rightarrow K_{i}$, $\alpha \mapsto \alpha e_{i}$,
and we have $K$-algebra isomorphisms $A_{i} \cong \Mat_{n_{i}}(D_{i})$ where $D_{i}$ is a skew field with $Z(D_{i}) \cong K_{i}$ 
(see \cite[(3.28)]{curtisandreiner_vol1}).
The Wedderburn decomposition (\ref{eq:weddecomp}) induces decompositions
\begin{equation}\label{eq:central-idem-decomps}
C = K_{1} \oplus \cdots \oplus K_{r}
\quad \textrm{and} \quad  \mathcal{O}_{C} = \mathcal{O}_{K_{1}} \oplus \cdots \oplus \mathcal{O}_{K_{r}},
\end{equation}
where $\calO_{K_i}$ denotes the ring of algebraic integers of $K_{i}$.
By \cite[(10.5)]{Reiner2003} we have $e_{1}, \ldots, e_{r} \in \mathcal{M}$ and each $\mathcal{M}_{i}:=\mathcal{M}e_{i}$ is a maximal $\mathcal{O}$-order 
(and thus a maximal $\mathcal{O}_{K_{i}}$-order) in $A_{i}$.
Moreover, each $\frf_{i}:=\frf e_{i}$ is a full two-sided ideal of $\calM_{i}$, 
each $\frg_{i} := \frg e_{i}$ is a nonzero integral ideal of $\mathcal{O}_{K_{i}}$, 
and we have decompositions
\begin{equation}\label{eq:max-and-ideal-decomp}
\mathcal{M} =  \mathcal{M}_{1} \oplus \cdots \oplus \mathcal{M}_{r},
\quad 
\frf = \frf_{1} \oplus \cdots \oplus \frf_{r}
\quad \textrm{and} \quad
\frg = \frg_{1} \oplus \cdots \oplus \frg_{r}.
\end{equation}

The reduced norm map $\nr : A \rightarrow C$ is defined componentwise (see \cite[\S 9]{Reiner2003})
and restricts to a group homomorphism $\nr : \mathcal{M}^{\times} \rightarrow \mathcal{O}_{C}^{\times}$.

\begin{lemma}\label{lem:defn-of-overline-nr}
There exists a surjective group homomorphism
$\overline{\nr}: (\mathcal{M}/\mathfrak{f})^{\times} \longrightarrow (\mathcal{O}_{C}/\mathfrak{g})^{\times}$
that fits into the commutative diagram
\[
\xymatrix@1@!0@=48pt { 
\mathcal{M}^{\times} \ar[rr]^{\nr}  \ar[d] & & \mathcal{O}_{C}^{\times} \ar[d] \\
(\mathcal{M}/\mathfrak{f})^{\times} \ar[rr]^{\overline{\nr}}  & &
(\mathcal{O}_{C}/\mathfrak{g})^{\times},
}
\]
where the vertical maps are induced by the canonical projections.
\end{lemma}

\begin{proof}
The decompositions \eqref{eq:weddecomp}, \eqref{eq:central-idem-decomps}, 
\eqref{eq:max-and-ideal-decomp} and the componentwise definition of the reduced
norm mean that we can and do assume without loss of generality that $r=1$, that is, 
$A$ is simple, $\mathcal{M}=\mathcal{M}_{1}$, $\mathfrak{f}=\mathfrak{f}_{1}$, 
$\mathfrak{g}=\mathfrak{g}_{1}$, and $C=K=K_{1}$.
Let $\mathfrak{p}$ be a prime ideal of $\mathcal{O}_{K}$ dividing $\mathfrak{g}$ and let $K_{\mathfrak{p}}$ denote the completion (not localisation) of $K$ at $\mathfrak{p}$.
If $M$ is an $\mathcal{O}_{K}$-module or an $\mathcal{O}_{K}$-algebra, then we write 
$\widehat{M}_{\mathfrak{p}} := M \otimes_{\mathcal{O}_{K}} \mathcal{O}_{K_{\mathfrak{p}}}$. 
By \cite[(7.6),(11.6)]{Reiner2003},
$\widehat{\mathcal{M}}_{\mathfrak{p}}$ is a maximal $\mathcal{O}_{K_{\mathfrak{p}}}$-order in 
the central simple $K_{\mathfrak{p}}$-algebra $A \otimes_{K} K_{\mathfrak{p}}$.
Let $\nr_{\mathfrak{p}} : \widehat{\mathcal{M}}_{\mathfrak{p}}^{\times} \rightarrow \mathcal{O}_{K_{\mathfrak{p}}}^{\times}$ denote the restriction of the reduced norm map.
Then by \cite[Corollary 2.4]{bley-boltje} we have that
$\nr_{\mathfrak{p}}(1+\widehat{\mathfrak{f}}_{\mathfrak{p}})=1+\widehat{\mathfrak{g}}_{\mathfrak{p}}$ and 
$\nr_{\mathfrak{p}}(\widehat{\mathcal{M}}_{\mathfrak{p}}^{\times})=\mathcal{O}_{K_{\mathfrak{p}}}^{\times}$.
Hence we have a commutative diagram
\[
\xymatrix@1@!0@=48pt { 
1  \ar[r] & 1+\widehat{\mathfrak{f}}_{\mathfrak{p}} \ar[rr] \ar[d]_{\nr_{\mathfrak{p}}} & &
\widehat{\mathcal{M}}_{\mathfrak{p}}^{\times} \ar[rr]  
\ar[d]_{\nr_{\mathfrak{p}}} & & 
(\widehat{\mathcal{M}}_{\mathfrak{p}}/\widehat{\mathfrak{f}}_{\mathfrak{p}})^{\times}
\ar[r] \ar[d]_{\overline{\nr}_{\mathfrak{p}}} &  1 \\
1 \ar[r] &1+\widehat{\mathfrak{g}}_{\mathfrak{p}} \ar[rr] & &
\mathcal{O}_{K_{\mathfrak{p}}}^{\times} \ar[rr]
& & ( \mathcal{O}_{K_{\mathfrak{p}}} / \widehat{\mathfrak{g}}_{\mathfrak{p}})^{\times} \ar[r] & 1,
}
\]
where $\overline{\nr}_{\mathfrak{p}}$ is induced by the other two vertical maps and
is surjective by the snake lemma.

The Chinese remainder theorem gives canonical isomorphisms
\[
(\mathcal{M}/\mathfrak{f})^{\times} \cong \prod_{\mathfrak{p} \mid \mathfrak{g}}
(\widehat{\mathcal{M}}_{\mathfrak{p}}/\widehat{\mathfrak{f}}_{\mathfrak{p}})^{\times}
\quad \textrm{ and } \quad
(\mathcal{O}_{K}/\mathfrak{g})^{\times} \cong \prod_{\mathfrak{p} \mid \mathfrak{g}}
(\mathcal{O}_{K_{\mathfrak{p}}} / \widehat{\mathfrak{g}}_{\mathfrak{p}})^{\times}.
\]
Let $\overline{\nr} = \prod_{\mathfrak{p} \mid \mathfrak{g}} \overline{\nr}_{\mathfrak{p}}$
and observe that the desired result now follows since 
the reduced norm map commutes with completion by \cite[(9.29)]{Reiner2003}.
\end{proof}

Let $\SL(\mathcal{M})=\ker(\nr: \mathcal{M}^{\times} \rightarrow \mathcal{O}_{C}^{\times})$ and 
$\SL(\mathcal{M}/\mathfrak{f})=\ker(\overline{\nr})$. 
(Note that in the literature the set $\SL(\mathcal{M})$, which is the group of units of reduced norm one, is sometimes also denoted by $\mathcal{M}^1$.)
Then by Lemma \ref{lem:defn-of-overline-nr} and the definitions,
we have the following commutative diagram
\begin{equation}\label{f1 diagram}
\xymatrix@1@!0@=48pt { 
1 \ar[r] & \SL(\mathcal{M}) \ar[rr] \ar[d]^{f_{1}} & &  \mathcal{M}^{\times} \ar[rr]^{\nr}  
\ar[d]_{f} & & \nr(\mathcal{M}^{\times}) \ar[r] \ar[d]_{f_{2}} &  1 \\
1 \ar[r] & \SL(\mathcal{M}/\mathfrak{f}) \ar[rr] & & (\mathcal{M}/\mathfrak{f})^{\times}  \ar[d]^{\pi} \ar[rr]^{\overline{\nr}}  
& & ( \mathcal{O}_{C}/\mathfrak{g})^{\times} \ar[d]^{\pi_2} \ar[r] & 1 \\
&                                  & &
{(\mathcal{M} / \mathfrak{f} )^{\times}} / {(\Lambda / \mathfrak{f} )^{\times}}& &  
\sfrac{(\mathcal{O}_{C} / \mathfrak{g})^{\times}}{\, \overline{\nr} (( \Lambda/\mathfrak{f} )^{\times})},
}
\end{equation}
where the rows are exact and the vertical maps are induced by the canonical projections.
Note that all the maps are group homomorphisms, apart from $\pi$, which is only a map of sets in general.

\begin{theorem}\label{thm:main-thm-for-freeness}
Let $X$ be a left ideal of $\Lambda$ such that $X + \mathfrak{f} = \Lambda$.
Suppose that there exists $\beta \in \mathcal{M}$ such that 
$\mathcal{M}X = \mathcal{M}\beta$. Then the following statements hold.
\begin{enumerate}
\item If $X$ is free over $\Lambda$, then
$\pi_{2}(\overline{\nr}(\overline{\beta}))$ is in the image of $\pi_{2} \circ f_{2}$.
\item If $f_{1}$ is surjective, then the converse of \textup{(a)} holds. More precisely, if $u \in \calM^\times$ and
$\overline{a} \in (\Lambda/\mathfrak{f})^{\times}$ satisfy
$\overline{\nr}(\overline{\beta}) = \overline{\nr}(\overline{u}) \overline{\nr}(\overline{a})$, 
then for any $v \in \SL(\mathcal{M})$ with $f_{1}(v)=\overline{\beta}\overline{ a}^{-1}\overline{u}^{-1}$, 
we have $X = \Lambda\alpha$ where $\alpha := (vu)^{-1}\beta$.
\end{enumerate}
\end{theorem}

\begin{proof}
(a) Suppose that $X$ is free over $\Lambda$.
Then there exists $\alpha \in X$ such that $X=\Lambda\alpha$.
Thus $\mathcal{M}\beta = \mathcal{M}X=\mathcal{M}(\Lambda\alpha)=\mathcal{M}\alpha$
and so there exists
$u \in \mathcal{M}^{\times}$ such that $\alpha = u\beta$.
Hence $\pi(\overline{\beta}) = \pi(\overline{u^{-1}})$ by
Proposition \ref{prop:equiv-conditions}.
In other words, there exists $a \in \Lambda$ such that
$\overline{a} \in (\Lambda/\mathfrak{f})^{\times}$
and $\overline{\beta}=\overline{u^{-1}}\overline{a}$. Thus
\[
\overline{\nr}(\overline{\beta}) 
= \overline{\nr}(\overline{u^{-1}}\overline{a})
= \overline{\nr}(\overline{u^{-1}})\overline{\nr}(\overline{a})
= f_{2}(\nr(u^{-1}))\overline{\nr}(\overline{a}),
\]
and so $\pi_{2}(\overline{\nr}(\overline{\beta}))=\pi_{2}(f_{2}(\nr(u^{-1})))$.

(b) Suppose that $f_1$ is surjective and $\pi_{2}(\overline{\nr}(\overline{\beta}))$
is in the image of $\pi_{2} \circ f_{2}$.
Then there exist $u,a$ and $v$ as in (b)  and so $\pi(\overline{\beta})= \pi(\overline{vu})$.
Hence by Proposition \ref{prop:equiv-conditions} we have 
$X = \Lambda \alpha$ where $\alpha := (vu)^{-1}\beta$.
\end{proof}

\section{Preliminaries on complexity}\label{sec:prelim-complexity}
We briefly recall the conventions that we will use for the complexity analysis of our algorithms. For details we
refer the reader to Lenstra~\cite{Lenstra1992} or Cohen~\cite[\S 1.1]{MR1228206}.

Let $l$ be the size of the input data measured by the number of required bits. Then an algorithm
is \textit{polynomial time} if the running time is $O(P(l))$ for a polynomial $P$.
An algorithm is \textit{subexponential time} if there exists $0 \le a < 1$ and
$b \in \R_{>0}$ such that the running time is $O(\exp(b \cdot l^a(\log l)^{a-1}))$.

A~\textit{probabilistic} algorithm may call a random number generator.
In this case we say that the algorithm is \textit{probabilistic polynomial time} if the expected running time is
$O(P(l))$ for a polynomial $P$. 
We adopt the same convention for probabilistic subexponential time algorithms.

Given two computational problems $\mathsf{A}$ and $\mathsf{B}$, a
\textit{(probabilistic) polynomial-time reduction} from $\mathsf{A}$ to
$\mathsf{B}$ is an algorithm that solves $\mathsf{A}$ using a polynomial
number of calls to an oracle solving $\mathsf{B}$ and is (probabilistic)
polynomial time outside of those calls to the oracle. 

Let $K$ be a number field with ring of integers $\calO = \calO_K$. We follow the convention of~\cite{Lenstra1992}
for representing our input data. In some more detail, if $V = K^{n}$
is an $n$-dimensional vector space over $K$,
we represent an $\calO$-module $M \subseteq V$ by a pseudobasis of $M$, that is, by elements $v_1,\dotsc,v_k \in V$ and
fractional $\calO$-ideals $\mathfrak a_1,\dotsc,\mathfrak a_k$ such that $M = \mathfrak a_1 v_1 \oplus \dotsb \oplus \mathfrak a_k v_k$ for some $k \leq n$. 
A~$K$-algebra $A$ of dimension $d$ is represented as a $d$-dimensional vector space together with the $K$-linear multiplication
map $A \otimes_K A \to A$, which is represented using $d^{3}$ elements of $K$.
Given a finite-dimensional $K$-algebra $A$, an $\calO$-order of $A$ is represented using a pseudobasis.
An $n$-dimensional $A$-module $V$ is represented as a vector space over $K$
together with $d$ matrices in $\Mat_{n}(K)$, one for each basis element
of $A$ describing the action on elements of $V$. Given an $\calO$-order
$\Lambda$, a $\Lambda$-lattice is represented by an $\calO$-submodule of a
finite-dimensional $A$-module, invariant under the action of $\Lambda$.

We will be mainly interested in solving the following two problems.

\begin{problem*}[$\pisomorphic$] 
Given a finite-dimensional $K$-algebra $A$, an $\mathcal{O}$-order $\Lambda$ in $A$,
and two $\Lambda$-lattices $X$ and $Y$, decide whether $X$ and $Y$ are isomorphic, and if so, 
return an isomorphism $X \to Y$.
\end{problem*}

\begin{problem*}[$\ppip$] 
Given a finite-dimensional $K$-algebra $A$, an $\mathcal{O}$-order $\Lambda$ in $A$,
and a full $\Lambda$-lattice $X$ in $A$, decide whether there exists $\alpha \in X$ such that
$X = \Lambda \alpha$, and if so, return such an element $\alpha$.
\end{problem*}

Under the assumption that $A$ satisfies hypothesis (H), 
we will reduce these questions to well-studied problems in algorithmic number theory.
These include $\ppip$ in the case where $A = K$ and $\Lambda=\mathcal{O}$, 
as well as the following problems:

\begin{itemize}
\item $\pfactor$: 
Given an ideal or element of the ring of integers $\mathcal{O}_{F}$ of a number field $F$, determine its factorisation into prime ideals.
\item $\pprimitive$: Given a finite field $\F_{q}$, determine $\alpha \in \F_{q}^{\times}$ 
such that $\F_{q}^{\times} = \langle \alpha \rangle$.
\item $\pdlog$: Given a finite field $\F_{q}$ and $\alpha, \beta \in \F_{q}^{\times}$ with
$\F_{q}^{\times} = \langle \alpha \rangle$, determine $n \in \Z_{\geq 0}$ such that $\alpha^{n} = \beta$.
\item $\punit$: Given the ring of integers $\calO_{F}$ of a number field $F$, 
determine a system of fundamental units for $\calO_{F}^{\times}$.
\end{itemize}

We will use the following standard convention and notation to denote variations
and instances of computational problems.
For example, for an $\mathcal{O}$-order $\Lambda$, we denote by $\pisomorphic_{\Lambda}$ the set of instances of
$\pisomorphic$ restricted to $\Lambda$-lattices. 
Similarly, 
we use $\ppip_\Lambda$ for the set of instances of $\ppip$ for $\Lambda$-lattices. 
Moreover, given a $\Lambda$-lattice $X$,
we use $\ppip(X)$ to denote the instance of $\ppip$ for the lattice $X$.
Note that in this case we still consider $\Lambda$ part of the input.

\begin{remark}\label{remark:subcomp}
Currently, the following complexity statements are known.
\begin{enumerate}
\item The problem $\pfactor_{\Z}$ can be solved in probabilistic subexponential time (\cite[Theorem 10.5]{MR1137100}). 
Given an ideal $I$ of the ring of integers $\mathcal{O}_{F}$ of a number field $F$,
the prime ideals of $\mathcal{O}_{F}$ lying above the rational prime factors of $\mathrm{Norm}_{F/\Q}(I)$
can be determined in probabilistic polynomial time (\cite[\S 6.2]{MR1228206});
hence there is a probabilistic polynomial-time reduction from $\pfactor$ to $\pfactor_{\Z}$. 
Since $\pprimitive(\F_q)$ is probabilistic polynomial-time reducible to \mbox{$\pfactor(q-1)$},
the same holds for $\pprimitive$.
Moreover, $\pdlog$ can be solved in subexponential time; see \cite{MR1759614} and the references therein.
While it is conjectured that $\ppip$ and $\unit$ can also be solved in subexponential time,
so far this has been established only under additional hypotheses and heuristics, including the Generalised Riemann Hypothesis (GRH); see ~\cite{MR1104698,MR3290946,MR3240816}.
\item There exist quantum polynomial-time algorithms for solving each of the problems $\pfactor_{\Z}$, 
$\pdlog$, $\pprimitive$ (\cite[\S{}5, \S{}6]{MR1471990}), $\unit$ (\cite[Theorem 1.2]{MR3238955}) and 
$\ppip$ for rings of integers of number fields (\cite[Theorem 1.3]{MR3478440}). 
\end{enumerate}
\end{remark}

Our approach to solving $\pisomorphic$ and $\ppip$ for noncommutative algebras $A$ satisfying 
hypothesis (H) relies crucially on the solution of the following two subproblems.

\begin{problem*}[$\pwedderburn$] 
Given a number field $K$ and a finite-dimensional semisimple $K$-algebra $A$ satisfying hypothesis \textup{(H)},
determine number fields $K_{i}$, integers $r, n_{i} \in \Z_{>0}$ and an explicit isomorphism 
$A \cong \prod_{i=1}^{r} \Mat_{n_i}(K_i)$.
\end{problem*}

\begin{problem*}[$\psplitting$] 
Given a number field $K$ and a split central simple $K$-algebra $A$, 
determine an isomorphism $A \cong \Mat_{n}(K)$ for some $n \in \Z_{>0}$.
\end{problem*}

\begin{remark}\label{remark:need_splitting}
For a finite-dimensional semisimple $K$-algebra $A$, an explicit decomposition 
$A \cong \prod_{i=1}^{r} A_{i}$ into simple $K$-algebras $A_{i}$, as well as the centre $K_{i}$ 
  of each $A_{i}$, can be computed in polynomial time by~\cite[1.5~B]{Friedl1985}.
Thus $\pwedderburn$ reduces to $\psplitting$.
The decision problem of checking whether
$A_{i} \cong \Mat_{n_{i}}(K_{i})$ for some $n_{i} \in \Z_{>0}$
is polynomial-time reducible to the computation of (the discriminant of)
  a maximal order in $A_{i}$ by~\cite[Corollary 3.4]{Nebe2009},
hence to $\pfactor$ by \cite[Corollary 5.3]{MR1246219}.
The problem of finding an explicit isomorphism appears to be a much harder problem.
In~\cite[Theorem 1]{MR2879232} it was shown that for algebras of bounded dimension over a fixed number field,
$\psplitting$ is probabilistic polynomial-time reducible to the problem of
computing a maximal order, hence to $\pfactor$.
\end{remark}

\section{Complexity of algorithms related to orders and their lattices}

Let $K$ be a number field with ring of integers $\calO = \calO_{K}$.
The aim of this section is to establish the complexity of certain algorithms
related to $\mathcal{O}$-orders and their lattices.
These algorithms have already appeared in the literature, either implicitly or explicitly,
but with either no or only partial analysis of their complexity.

\subsection{Computing maximal orders}\label{subsec:comp-max-orders}

Let $A$ be a finite-dimensional semisimple $K$-algebra and let $\Lambda$ be an $\mathcal{O}$-order
in $A$.
Let $d = \dim_{K}A$ and let $\tr \colon A \rightarrow K$ denote the reduced trace map 
(see \cite[\S 7D]{curtisandreiner_vol1}).
Following \cite[\S 26A]{curtisandreiner_vol1},
we define $\disc(\Lambda)$ to be the ideal of  $\calO$ generated by all elements
\[
\det\left( \tr(x_ix_j)_{1 \le i,j \le d}\right) \text{ with } x_1, \ldots, x_d \in \Lambda.
\]

By applying a result of \cite{Friedl1985}, the following is straightforward to deduce from the results of \cite{Friedrichs} 

\begin{prop}\label{prop:comp-max-order}
Let $\Lambda$ and $A$ be as above. 
Then the problem of computing a maximal $\mathcal{O}$-order $\mathcal{M}$ in $A$
containing $\Lambda$ is probabilistic polynomial-time reducible
to $\pfactor(\disc(\Lambda))$.
\end{prop}

\begin{proof}
Let $\mathfrak{p}$ be a maximal ideal of $\calO$ dividing $\disc(\Lambda)$ and 
write $v_\mathfrak{p}(-)$ for the $\mathfrak{p}$-adic valuation.
It follows from \cite[(3.17)]{Friedrichs}
that the computation of an order $\Lambda^{(\mathfrak{p})}$ such that 
$v_\mathfrak{p}([\Lambda^{(\mathfrak{p})} : \Lambda ]_\calO)$ is maximal reduces in polynomial time to the problem of computing the maximal two-sided ideals of an order containing $\mathfrak{p}$.
Now fix an order $\Gamma$. Then  the maximal two-sided ideals of $\Gamma$ containing $\mathfrak{p}$ are the preimages of the maximal two-sided ideals under the canonical projection $\Gamma \to (\Gamma/\mathfrak{p}\Gamma)/{\jac(\Gamma/\mathfrak{p}\Gamma)}$ (see \cite[(5.23)]{Friedrichs}). 
As a decomposition of this $(\calO/\mathfrak{p})$-algebra into simple components and therefore the maximal two-sided ideals can be found in probabilistic polynomial time by~\cite[1.5~B]{Friedl1985}, an order $\Lambda^{(\mathfrak{p})}$ can be determined in probabilistic polynomial time.
By~\cite[(3.19)]{Friedrichs} the order $\sum_{\mathfrak{p}} \Lambda^{(\mathfrak{p})}$ is maximal, where 
$\mathfrak{p}$ runs over the maximal ideals of $\mathcal{O}$ dividing $\disc(\Lambda)$. 
Therefore the computation of a maximal order $\calM$ containing $\Lambda$ reduces in probabilistic polynomial time to $\pfactor(\disc(\Lambda))$.
\end{proof}

\subsection{Nice maximal orders}\label{subsec:nice}

Let $n \in \Z_{>0}$ and let $A = \Mat_{n}(K)$ be a full matrix algebra.
For a nonzero fractional ideal $\mathfrak{a}$ of $\calO$ let 
\[
\calM_{\mathfrak a, n} :=
\left(
  \begin{array}{cccc}
    \calO & \ldots & \calO & \mathfrak{a}^{-1} \\
    \vdots & \ddots & \vdots & \vdots \\
    \calO & \ldots & \calO & \mathfrak{a}^{-1} \\
    \mathfrak{a}   & \ldots & \mathfrak{a}   & \calO
  \end{array} \right)
\]
denote the $\mathcal{O}$-order in $A$ consisting of all $n \times n$ matrices $(x_{ij})_{1 \leq i,j \leq n}$ 
where $x_{11}$ ranges over all elements of $\mathcal{O}$, \ldots, $x_{1n}$ ranges over all elements of
$\mathfrak{a}^{-1}$, and so on. (In the case $n=1$, we take $\calM_{\mathfrak a, n}=\mathcal{O}$.)
We say that a maximal $\mathcal{O}$-order in $A$ is \textit{nice} if it is equal to 
$\calM_{\mathfrak a, n}$ for some choice of $\mathfrak{a}$.
By \cite[(27.6)]{Reiner2003} every maximal $\calO$-order in $A$ is conjugate to 
a nice maximal order.

\begin{lemma}\label{lem:nice}
There exists a probabilistic polynomial-time algorithm that given a maximal $\calO$-order $\calM$
in $A=\Mat_{n}(K)$, determines a nonzero fractional ideal $\mathfrak{a}$ of $\calO$ and $S \in \GL_{n}(K)$
such that $S\calM S^{-1} = \calM_{\mathfrak a, n}$.
\end{lemma}

\begin{proof}
The algorithm is presented in~\cite[\S 5]{BJ08} and works by reducing the problem to the computation of a Steinitz form of an $\calO$-lattice of rank $n$, which can be performed in probabilistic polynomial time by Corollary~\ref{cor:steinitz}. 
\end{proof}

\subsection{Norm equations and principal ideals}
Let $r \in \Z_{>0}$ and let $A = \prod_{i=1}^{r} \Mat_{n_i}(K_i)$
where $K_{i}$ is a finite field extension of $K$ and $n_{i} \in \Z_{>0}$ for each $i$.
In particular, $A$ is a finite-dimensional semisimple $K$-algebra satisfying hypothesis (H).
Let $C$ be the centre of $A$, which we can and do identify with $\prod_{i=1}^{r} K_{i}$.
Let $\calM$ be a maximal $\mathcal{O}$-order in $A$ and 
let $\mathcal{O}_{C} = \mathcal{M} \cap C = \prod_{i=1}^{r} \mathcal{O}_{K_{i}}$.

\begin{lemma}\label{lem:normeq}
The reduced norm map $\nr \colon \calM^\times \to \OC^\times$ is surjective.
Moreover, there exists a probabilistic polynomial-time algorithm that given
$\calM$ and $a \in \OC^\times$ determines $\alpha \in \mathcal M^\times$ such that $\nr(\alpha) = a$.
\end{lemma}

\begin{proof}
By decomposing $\calM$ using the central primitive idempotents of $A$, it suffices to consider the case
$A = \Mat_{n}(K)$, in which we must have $\mathcal{O}_{C} = \mathcal{O}$.
Then the reduced norm map $\nr: A \rightarrow K$ is just the usual determinant map.
Moreover, using Lemma~\ref{lem:nice}, we can and do assume that $\calM = \calM_{\mathfrak a, n}$
is a nice maximal order.
Since $\alpha = \operatorname{diag}(a,1,\dotsc,1) \in \calM_{\mathfrak a, n}^{\times}$ satisfies
$\nr(\alpha) = a$, the claim follows.
\end{proof}

An algorithm for solving the principal ideal problem for $\calM$-lattices was given in~\cite[\S 5]{BJ08}.
We now analyse its complexity.

\begin{prop}\label{prop:pipmaximal}
The problem $\ppip_{\calM}$ is probabilistic polynomial-time reducible to one instance of
$\ppip_{\mathcal O_{K_i}}{}$ for each $i = 1,\dotsc,r$.
\end{prop}

\begin{proof}
By decomposing $\calM$ using the central primitive idempotents of $A$, it suffices to consider the case
$A = \Mat_{n}(K)$, in which we must have $\mathcal{O}_{C} = \mathcal{O}$.
Let $X$ be a full $\calM$-lattice in $A$.
Let $e_{11} \in A$ be the matrix with the top left entry equal to $1$ and all other entries equal to $0$.
Using Lemma~\ref{lem:nice}, we can and do assume that $\calM = \calM_{\mathfrak a, n}$ is a nice maximal order.
By~\cite[Corollary 5.4]{BJ08}, it is sufficient to check whether the Steinitz class of the $\calO$-module $e_{11}X$ is equal to $[\mathfrak a^{-1}]$, which amounts to testing whether a certain ideal of $\calO$ is principal.
\end{proof}

\subsection{Computing isomorphisms between localised lattices}
Let $\Lambda$ be an $\calO$-order in a finite-dimensional $K$-algebra $A$.
Given two $\Lambda$-lattices $X$ and $Y$ and a maximal ideal $\mathfrak{p}$ of $\calO$,
we wish to determine whether there exists an isomorphism $X_\mathfrak p \cong Y_\mathfrak{p}$ of 
$\Lambda_{\mathfrak{p}}$-lattices and to compute such an isomorphism if so.
By computing an isomorphism we mean computing a $\Lambda$-morphism $f\colon X \to Y$ such that 
its localisation
$f_{\mathfrak p} \colon X_\mathfrak p \to Y_\mathfrak p$ is an isomorphism.

We first consider the case where $X$ is a full $\Lambda$-lattice in $A$ and $Y = \Lambda$,
for which an algorithm was presented in~\cite[\S 4.2]{BW09} (although the algorithm was presented only in the context of semisimple algebras, the semisimplicity hypothesis is in fact unnecessary).
We now outline the algorithm and analyse its complexity.

\begin{prop}\label{prop:local_freeness}
There exists a probabilistic polynomial-time algorithm that given $A$, $\Lambda$ 
and $\mathfrak{p}$ as above and 
a full $\Lambda$-lattice $X$ in $A$, decides whether $X_{\mathfrak p}$ is free over $\Lambda_\mathfrak{p}$,
and if so, returns $\alpha \in X$ such that $X_\mathfrak{p} = \Lambda_\mathfrak{p}  \alpha$.
\end{prop}

\begin{proof}
Consider the finitely generated $\calO/\mathfrak{p}$-algebra 
$R_{\mathfrak{p}} := 
\Lambda/\mathfrak{p} \Lambda 
\cong \Lambda_{\mathfrak{p}} / \mathfrak{p} \Lambda_{\mathfrak{p}}$.
It follows from \cite[1.5~A]{Friedl1985}, that the Jacobson radical
$J_{\mathfrak{p}} = \jac(R_{\mathfrak{p}})$ can be determined in polynomial time.
Let $\overline{R}_{\mathfrak{p}} = R_{\mathfrak{p}}/J_{\mathfrak{p}}$. 
By Lemma \ref{lem:cyclic-iff-free} and Nakayama's lemma,
 $X_{\mathfrak{p}}$ is free over $\Lambda_{\mathfrak{p}}$ if and only if 
$\overline{X}_{\mathfrak{p}} 
:= 
(X/\mathfrak{p}X)/J_{\mathfrak{p}}(X/\mathfrak{p}X)
\cong 
(X_{\mathfrak{p}}/\mathfrak{p}X_{\mathfrak{p}})/J_{\mathfrak{p}}(X_{\mathfrak{p}}/\mathfrak{p}X_{\mathfrak{p}})$ is free of rank $1$ 
over $\overline{R}_{\mathfrak{p}}$.
Using algorithms of Friedl--Rónyai~\cite[1.5~B]{Friedl1985} and Ronyai~\cite[Theorem 6.2]{Ronyai1987},
one can determine an isomorphism of $\overline{R}_{\mathfrak{p}}$ 
with a product of matrix algebras over finite fields $k_{i}$ in probabilistic polynomial time.
The final steps are just linear algebra over finite fields.
\end{proof}

The following algorithm without the complexity statement was given in 
 \cite[\S 8.4]{MR4136552}.

\begin{corollary}\label{cor:local_freeness}
There exists a probabilistic polynomial-time algorithm that given $A$, $\Lambda$ 
and $\mathfrak{p}$ as above and $\Lambda$-lattices $X$ and $Y$, 
decides whether $X_\mathfrak{p}$ and $Y_\mathfrak{p}$ are isomorphic as $\Lambda_{\mathfrak{p}}$-lattices, and if so, returns $f \in \Hom_\Lambda(X, Y)$ such that the localisation 
$f_{\mathfrak{p}} \colon X_\mathfrak{p} \to Y_\mathfrak{p}$ is an isomorphism.
\end{corollary}

\begin{proof}
We use Proposition~\ref{prop:local_freeness} together with the reduction to the free rank $1$ case
given by Proposition~\ref{prop:hom-free-over-end}. 
Both $\End_{\Lambda}(Y)$ and $\Hom_{\Lambda}(X, Y)$ can be determined 
as described in \cite[\S 7.3]{MR4136552} using 
pseudo-Hermite normal form and pseudo-Smith normal form computations,
  which are probabilistic polynomial time by~\cite[Theorem~34, Proposition~43]{Biasse2016}.
Using Proposition~\ref{prop:local_freeness} one can determine in probabilistic polynomial time whether the
$(\End_\Lambda(Y))_\mathfrak{p}$-lattice $(\Hom_\Lambda(X, Y))_\mathfrak{p}$ is free of rank $1$. 
If not, then $X_\mathfrak{p}$ and $Y_\mathfrak{p}$ are not isomorphic over $\Lambda_{\mathfrak{p}}$. 
If so, then the algorithm
returns a free generator $f \in  \Hom_\Lambda(X, Y)$ of $(\Hom_\Lambda(X, Y))_\mathfrak{p}$
over $(\End_\Lambda(Y))_\mathfrak{p}$. 
Then $X_\mathfrak{p} \cong Y_\mathfrak{p}$ over $\Lambda_{\mathfrak{p}}$
if and only if the localisation $f_{\mathfrak{p}} \colon X_\mathfrak{p} \to Y_\mathfrak{p}$ 
is an isomorphism.
\end{proof}

\begin{remark}
Given two $\Lambda$-lattices $X$ and $Y$, Corollary~\ref{cor:local_freeness} can be used to decide
if $X$ and $Y$ are in the same genus, that is, whether $X_{\mathfrak{p}}$ and $Y_{\mathfrak{p}}$ are 
isomorphic $\Lambda_{\mathfrak{p}}$-lattices for every nonzero prime ideal $\mathfrak{p}$ of $\calO$.
Note that a necessary condition is that $KX$ and $KY$ are isomorphic as $A$-modules.
By \cite[Corollary 3]{MR1809971} there is a polynomial-time algorithm that decides 
whether $KX$ and $KY$ are isomorphic as $A$-modules, and if so, computes an isomorphism; 
hence the problem reduces to the case $KX = KY$.
In this situation, $X$ and $Y$ are in the same genus if and only if
$X_\mathfrak{p}$ and $Y_\mathfrak{p}$ are isomorphic $\Lambda_{\mathfrak{p}}$-lattices
for the finitely many prime ideals
$\mathfrak{p}$ dividing the module index $[X : Y]_\calO$ (see \cite[\S 3]{Frohlich1967}).
Hence checking whether $X$ and $Y$ are in the same genus is polynomial-time
reducible to $\pfactor([X : Y]_{\calO})$.
\end{remark}

\subsection{Finding a suitable choice of locally free left ideal}
\label{sec coprime cond}

Let $A$ be a finite-dimensional semisimple $K$-algebra and let $\Lambda$ be an $\calO$-order in $A$.
By \cite[(10.4)]{Reiner2003} there exists a (not necessarily unique) maximal $\mathcal{O}$-order $\mathcal{M}$ in $A$ containing $\Lambda$.
Let $\mathfrak{f}$ be any proper full two-sided ideal of $\mathcal{M}$ that is contained in $\Lambda$.
The following result without the complexity statements is a consequence of a special case
of the argument given in \cite[\S 5.1]{BJ11}.

\begin{prop}\label{prop:suitable-locally-free-lattice}
Given $A$, $\Lambda$ and $\mathfrak{f}$ as above, and a full $\Lambda$-lattice $X$ in $A$,
the problem of determining whether $X$ is locally free over $\Lambda$, and if so, 
computing an element $\xi \in A^{\times}$ such that $X \xi \subseteq \Lambda$ and $X\xi + \frf = \Lambda$,
is probabilistic polynomial-time reducible to $\pfactor(\calO \cap \frf)$.
\end{prop}

\begin{proof}
Let $\mathrm{MaxSpec}(\mathcal{O})$ denote the set of all maximal ideals of $\mathcal{O}$.
Let $\mathfrak{S}= \{ \frp_{1}, \ldots, \frp_{n} \}$ be the subset consisting 
of ideals that divide $\calO \cap \frf$ (note that this is a proper nonzero ideal of $\mathcal{O}$)
and let $\mathfrak{T} = \mathrm{MaxSpec}(\mathcal{O}) \setminus \mathfrak{S}$.
Observe that for every $\mathfrak{p} \in \mathfrak{T}$, we have $\frf_\frp = \Lambda_\frp=\mathcal{M}_{\frp}$
and so $X_{\mathfrak{p}}$ is free over $\Lambda_{\mathfrak{p}}$ by \cite[(18.10)]{Reiner2003}.
Moreover, 
for each $i$, checking whether $X_{\frp_i}$ is free over $\Lambda_{\frp_i}$, and if so, 
computing $\omega_{i} \in X$ such that $X_{\frp_i} = \Lambda_{\frp_i} \omega_{i}$,
can be performed in probabilistic polynomial time by Proposition~\ref{prop:local_freeness}.
In particular, if this step is completed successfully, then $X$ is locally free over $\Lambda$.

By \cite[Proposition 1.3.11]{MR1728313}, elements $\beta_{1}, \ldots, \beta_{n} \in \mathcal{O}$ such that for each $i$, we have
\[
\beta_i \equiv 1 \bmod{\frp_i} 
\quad \text{ and } \quad
\beta_i \equiv 0 \bmod{\frp_j} \text{ for } 1 \le j \le n, j \ne i,
\]
can be computed in polynomial time.
For each $i$, let $\nu_i \in \calO \setminus \frp_i$ be an element such that 
$X \omega_i^{-1} \nu_i \subseteq \Lambda$.
Then $X\xi \subseteq \Lambda$ where $\xi := \sum_{i=1}^{n} \beta_{i} \omega_{i}^{-1} \nu_{i}$. 
By construction we have $(X \xi)_{\mathfrak{p}_{i}} = \Lambda_{\mathfrak{p}_{i}}$ for each $i$.
Moreover,  $\frf_\frp = \Lambda_\frp$ for all $\frp \in \mathfrak{T}$.
Therefore  $(X\xi + \mathfrak{f})_{\mathfrak{p}} = \Lambda_{\mathfrak{p}}$ 
for all $\mathfrak{p} \in \mathrm{MaxSpec}(\mathcal{O})$, and so 
$X\xi + \frf = \Lambda$ by \cite[(4.21)]{Reiner2003}.
\end{proof}

\subsection{Computing generators of $(\Lambda/\mathfrak{f})^{\times}$ and $K_{1}(\Lambda/\mathfrak f)$}\label{sec:comp-gens}
We first recall some definitions from algebraic $K$-theory and refer the reader to 
\cite[\S 40]{curtisandreiner_vol2} for more details.
For any ring $R$, the Whitehead group $K_{1}(R)$ is defined as $\GL(R)/[\GL(R),\GL(R)]$, where $\GL(R) = \varinjlim \GL_n(R)$ and $\GL_n(R)$ embeds into $\GL_{n + 1}(R)$ via
\[ 
\alpha \mapsto \begin{pmatrix} \alpha & 0 \\ 0 & 1 \end{pmatrix}.
\]
In particular, there is a canonical map $R^\times \to \GL(R) \to K_1(R)$.

Now assume the notation and setting of \S \ref{sec coprime cond}.
Since $\Lambda/\mathfrak{f}$ is of finite cardinality, it is semilocal and so the canonical map
\[
(\Lambda/\mathfrak f)^{\times} \lra K_1(\Lambda/\mathfrak f)
\]
is surjective by \cite[(40.31)]{curtisandreiner_vol2}. 
We consider the problems of computing generators of $(\Lambda/\mathfrak{f})^{\times}$ and of
$K_{1}(\Lambda/\mathfrak f)$, where the latter task means computing elements $x_1,\dotsc,x_n \in (\Lambda/\mathfrak{f})^{\times}$ such that their images generate $K_1(\Lambda/\mathfrak f)$.

An algorithm for computing generators of $K_1(\Lambda/\mathfrak{f})$ is
described in \cite[\S 3.4--3.7]{bley-boltje}. 
With minor modifications, this algorithm also computes a generating set of $(\Lambda/\mathfrak f)^\times$.
In this subsection, we will analyse the complexity of both these algorithms.
To treat both cases simultaneously, for a ring $R$ we let $\mathrm{G}(R)$ denote either $K_{1}(R)$ or 
$R^{\times}$.

Let $C$ denote the centre of $A$ and let $\mathcal{O}_{C}$ be the integral closure of $\mathcal{O}$ in $C$.
Let $\frg = \mathfrak{f} \cap C$ and note that this is a proper full ideal of $\mathcal{O}_{C}$ and
of $\Lambda \cap \mathcal{O}_{C} = \Lambda \cap C$.
Let $\frg = \prod_{\mathfrak{P} \in \EuScript{P}} \mathfrak{P}^{e_\mathfrak{P}}$ be the
prime ideal decomposition of $\frg$ in $\mathcal{O}_{C}$, where the set $\EuScript{P}$
of prime ideals of $\mathcal{O}_{C}$ is defined by the decomposition.
Set $\EuScript{P}' := \{ \mathfrak P \cap \Lambda \mid \mathfrak{P} \in \EuScript{P} \}$, 
a set of prime ideals of $\Lambda \cap \mathcal{O}_{C}$.
For each $\mathfrak{p} \in \EuScript{P}'$ consider the ideal 
\[
\mathfrak{q} := 
\bigcap_{\substack{\mathfrak{P} \in  \EuScript{P},\\ \mathfrak P \cap \Lambda 
=
\mathfrak p}} (\mathfrak{P}^{e_\mathfrak{P}} \cap \Lambda). 
\]
We write $\EuScript{Q}$ for the set of ideals $\mathfrak{q}$.
Then by \cite[Proposition 3.2]{bley-endres}
\[
\mathfrak{g}
= \prod_{\mathfrak{q} \in \EuScript{Q}}  \mathfrak{q}
= \bigcap_{\mathfrak{q} \in \EuScript{Q}}  \mathfrak{q}
\]
is the unique primary decomposition of $\mathfrak{g}$ when considered as an ideal of
$\Lambda \cap \mathcal{O}_{C}$.
Moreover, by \cite[Lemma 3.5]{bley-boltje} we have
\[
\frf = \bigcap_{\frq \in \EuScript{Q}} (\frq\Lambda + \frf) = \prod_{\frq \in \EuScript{Q}} (\frq\Lambda + \frf)
\]
and by the Chinese remainder theorem we obtain an isomorphism
\[
\Lambda/\mathfrak f \cong \prod_{\mathfrak q \in \EuScript{Q}} \Lambda/(\mathfrak q\Lambda + \mathfrak f).
\]
This induces a decomposition
\[
\mathrm{G}(\Lambda/\mathfrak f) \cong
\prod_{\mathfrak{q} \in \EuScript{Q}} \mathrm{G}(\Lambda/(\mathfrak q\Lambda + \mathfrak f)).
\]
Thus, given $\EuScript{Q}$, it suffices to compute generators of
$\mathrm{G}(\Lambda/(\mathfrak{q}\Lambda + \mathfrak{f}))$ for each $\mathfrak{q} \in \EuScript{Q}$.

Now fix $\mathfrak{q} \in \EuScript{Q}$ and let $\mathfrak{p} = \mathfrak{P} \cap \Lambda \in \EuScript{P}'$
be the associated prime ideal of $\Lambda \cap \mathcal{O}_{C}$ for some $\mathfrak{P} \in \EuScript{P}$.
As shown in \cite[\S 3.7]{bley-boltje}, we have an exact sequence
\begin{equation}\label{eq:exact-seq-breakdown-gens}
  (1 + \mathfrak p \Lambda + \mathfrak f)/(1 + \mathfrak q \Lambda + \mathfrak f) \lra
  \G(\Lambda/(\mathfrak q \Lambda + \mathfrak f)) \lra
  \G(\Lambda/(\mathfrak p \Lambda + \mathfrak f)) \lra 1. 
\end{equation}
We consider the problems of computing generators for the first and third terms in this sequence.
Let $d := \dim_{K} A$.

\begin{lemma}\label{lem:compk1_1}
Given $\Lambda$, $\mathfrak{f}$, and $\mathfrak p$ as above, the problem of computing generators of
\[
(\Lambda/(\mathfrak p \Lambda + \mathfrak{f}))^{\times} 
\text{ or } 
K_{1}(\Lambda/(\mathfrak{p} \Lambda + \mathfrak{f}))
\]  
is probabilistic polynomial-time reducible to at most $d$ instances of the problem 
$\pprimitive$ for extensions of $\mathcal{O}/(\mathcal{O}\cap \mathfrak{P})$ of degree at most $d$.
The number of generators is at most $d([K : \Q]+2)$.
\end{lemma}

\begin{proof}
Let $k$ denote the finite field $\mathcal{O}/(\mathfrak{p} \cap \mathcal{O})$.
Let $R = \Lambda/(\mathfrak{p} \Lambda + \mathfrak {f})$ and note that this
is annihilated by $\mathfrak{p} \cap \mathcal{O}$. 
Thus $R$ is a $k$-algebra such that $\dim_{k} R \leq d$.
In particular, $R$ is artinian, so its Jacobson radical $J = \jac(R)$ is nilpotent by \cite[(5.15)]{curtisandreiner_vol1}.
Since we have a decreasing filtration $J \supseteq J^2 \supseteq \cdots \supseteq J^d$ and $\dim_k(J) \le d-1$ we obtain $J^{d}=0$.
By \cite[Lemma 3.6]{bley-boltje} and the same reasoning as in \cite[\S 3.7]{bley-boltje}, we have an exact sequence
\[
1 + J \lra \G(R) \lra \G(R/J) \lra 1.
\]
We first discuss the computation of generators for $1 + J$. 
To this end, let $l \in \Z_{\geq 0}$ be minimal subject to the condition $J^{2^l} = 0$ and note that
$2^{l} \leq 2d$.
Consider the filtration
\[
1 + J \supseteq 1 + J^2 \supseteq \dotsb \supseteq 1 + J^{2^{l - 1}} \supseteq 1.
\] 
Generators of $J$ can be determined in polynomial time using the algorithms of \cite[1.5~A]{Friedl1985}.
For each $i = 0,\dotsc,l-1$, the map $\overline{x} \mapsto \overline{x-1}$ induces an isomorphism
\[
(1 + J^{2^i})/(1 + J^{2^{i+1}}) \lra J^{2^i}/J^{2^{i+1}}
\]
of abelian groups and so it follows that we can find generators of $1 + J$ in polynomial time.
For each $i = 0, \dotsc, l - 1$ the number of generators of $J^{2^i}/J^{2^{i + 1}}$
is bounded by $\dim_k(J^{2^i}/J^{2^{i+1}}) [K : \Q]$. Now summing over $i = 0,\dotsc,l-1$ shows that
$1 + J$ is generated by at most $\dim_k(J)  [K : \Q] \leq (d-1) [K : \Q]$ elements.

Using algorithms of \cite[1.5~B]{Friedl1985} and \cite[Theorem 6.2]{Ronyai1987},
one can determine an isomorphism $R/J \cong \prod_{1 \leq i \leq r} \Mat_{n_{i}}(k_{i})$
with a product of matrix algebras over finite fields $k_{i}$ in probabilistic polynomial time.
Since
\[
\G(R/J) \cong \prod_{1 \leq i \leq r} \G(\Mat_{n_i}(k_i)), 
\]
this problem reduces to the computation of each $\G(\Mat_{n_i}(k_i))$, which we claim is generated by at most $2$ elements.
If $\G{} = K_1$, then the claim follows from the fact that the
canonical maps $k_{i}^{\times} \to K_{1}(k_{i}) \to K_1(\Mat_{n_i}(k_i))$ are isomorphisms.
If $\G{} = (-)^\times$, then the claim follows 
{from~\cite{Taylor1987}, where it is shown that given a primitive element of $k_i$ one can write down directly a two element generating set of $\GL_{n_i}(k_i)$.
Since $\dim_{k}(R/J) \leq d$ we have $r \leq d$ and $[k_i : k] \leq d$ for each $i$.
Finally, note that $\mathcal{O}/(\mathcal{O}\cap\mathfrak{p}) = \mathcal{O} / (\mathcal{O}\cap\mathfrak{P})$ since $\mathcal{O}\subseteq\Lambda$ implies $\mathcal{O}\cap\mathfrak{p}= \mathcal{O}\cap\mathfrak{P}\cap\Lambda = \mathcal{O}\cap\mathfrak{P}$}.
In particular, $G(R/J)$ is generated by at most $2r \leq 2d$ elements.
\end{proof}

\begin{lemma}\label{lem:compk1_2}
Given $\Lambda$, $\mathfrak{f}$, $\mathfrak{p}$, and $\mathfrak{q}$ as above, we set
\[
e_\frp = \max\{e_{\mathfrak{P}} \mid \mathfrak{P} \in \EuScript{P}, \ \mathfrak{P} \cap \Lambda = \mathfrak{p}\}.
\]
Then there exists a polynomial-time algorithm that returns $m$ elements of $\Lambda$ whose classes generate
$(1 + \mathfrak{p} \Lambda + \mathfrak{f})/(1 + \mathfrak{q} \Lambda + \mathfrak{f})$. If $e_\frp = 1$, we have $m=0$. If
$e_\frp > 1$, the number $m$ of generators is bounded by $d (1+\log_2(e_\frp)) [K:\Q]$.
\end{lemma}

\begin{proof}
If $e_\frp=1$, we clearly have $\frp = \frq$ and so $m=0$.
If $e_\frp > 1$, we let $l \in \Z_{>0}$ be minimal subject to the condition $\mathfrak{p}^{2^l} \subseteq \mathfrak{q}$.
Then there exists a filtration
\[
\mathfrak{p} \Lambda + \mathfrak{f} 
\supseteq (\mathfrak{q} + \mathfrak{p}^2)\Lambda + \mathfrak{f} 
\supseteq \dotsb \supseteq 
(\mathfrak{q} + \mathfrak{p}^{2^{l - 1}})\Lambda + \mathfrak{f} \supseteq \mathfrak{q} \Lambda + \mathfrak{f}.
\] 
For each $i = 0,\dotsc,l-1$, the map $\overline{x} \mapsto \overline{x-1}$ induces an isomorphism
\[
\frac{1 + (\mathfrak{q} + \mathfrak{p}^{2^i})\Lambda + \mathfrak{f}}
{1 + (\mathfrak{q} + \mathfrak{p}^{2^{i+1}})\Lambda + \mathfrak{f}}
\lra
\frac{(\mathfrak{q} + \mathfrak{p}^{2^i})\Lambda + \mathfrak{f}}{(\mathfrak{q} + \mathfrak{p}^{2^{i+1}})\Lambda + \mathfrak{f}}
\]
of abelian groups.
Hence any $\Z$-basis of the right hand side
yields generators of the left hand side.
It remains to bound $l$.
For every $\mathfrak{P} \in \EuScript{P}$ with $\mathfrak{P} \cap \Lambda = \mathfrak{p}$ the inclusion
$\mathfrak{p}^{e_\frp} = (\mathfrak{P} \cap \Lambda)^{e_\frp} \subseteq \mathfrak{P}^{e_\frp} \cap \Lambda \subseteq
\mathfrak{P}^{e_{\mathfrak{P}}} \cap \Lambda$ holds.
Hence $\mathfrak{p}^{e_\frp} \subseteq \mathfrak{q}$
and therefore $2^{l} \leq 2e_\frp$, which gives $l \leq 1 + \log_{2}(e_\frp)$.

Since any quotient
$((\mathfrak{q} + \mathfrak{p}^{2^i})\Lambda + \mathfrak{f}) / ((\mathfrak{q} + \mathfrak{p}^{2^{i+1}})\Lambda + \mathfrak{f})$
is generated by at most $d [K : \Q]$ elements, the quotient
$(1 + \frp \Lambda + \frf)/(1 + \frq \Lambda + \frf)$ is generated by at most
$(1 + \log_{2}(e_\frp))  d  [K : \Q]$ elements.
\end{proof}

\begin{prop}\label{prop:compk1}
Given $\Lambda$ and $\frf$ as above, the problem of computing generators of
$(\Lambda/\mathfrak{f})^{\times}$ and $K_{1}(\Lambda/\mathfrak{f})$ is
probabilistic polynomial-time reducible to the factorisation of 
$\mathfrak{g} := \mathfrak{f} \cap \mathcal{O}_{C}$ 
as an ideal of $\mathcal{O}_{C}$ and,
for each prime ideal divisor $\mathfrak{P}$ of $\mathfrak{g}$,
at most $d$ instances of \pprimitive{} for extensions of $\mathcal{O}/(\mathcal{O}\cap\mathfrak{P})$
of degree at most $d$.
The number of generators is bounded by $5 d [ K:\Q]\log_2|\OC/\frg|$.
\end{prop}

\begin{proof}
Using the factorisation of $\mathfrak{g} = \mathfrak{f} \cap \mathcal{O}_{C}$,
one can determine the sets of ideals 
$\EuScript{P}$, $\EuScript{P}'$ and $\EuScript{Q}$ in polynomial time.
Since $\lvert \EuScript{Q} \rvert = \lvert \EuScript{P}' \rvert \leq \lvert \EuScript{P} \rvert 
\leq \log_{2}\lvert \mathcal{O}_{C}/\frg \rvert$,
the claim follows from the reduction to the computation of
$\mathrm{G}(\Lambda/(\mathfrak{q}\Lambda + \mathfrak{f}))$ for each $\mathfrak{q} \in \EuScript{Q}$
discussed at the beginning of the section, the exact sequence
\eqref{eq:exact-seq-breakdown-gens}, Lemmas~\ref{lem:compk1_1} and~\ref{lem:compk1_2} and the following computation
\begin{eqnarray*}
  &&  d |\EuScript{P}'| ([K:\Q]+2) +   d [K:\Q]\sum_{\frp \in \EuScript{P}', e_\frp > 1}(1 + \log_2(e_\frp)) \\
  &\le& d [ K:\Q]\left( 3|\EuScript{P}'| + \sum_{\frp \in \EuScript{P}', e_\frp > 1}2\log_2(e_\frp) \right) \\
  &\le& d [ K:\Q]\left( 3|\EuScript{P}'| + 2\sum_{\frP \in \EuScript{P}}\log_2(e_\frP) \right) \\
  &\stackrel{(*)}\le& d [ K:\Q]\left( 3|\EuScript{P}'| + 2\log_2|\OC/\frg| \right) \\
   &\le& 5 d [ K:\Q]\log_2|\OC/\frg|.
\end{eqnarray*}
The inequality $(*)$ is a consequence of
\[
  \sum_{\frP \in \EuScript{P}}\log_2(e_\frP) \le
  \sum_{\frP \in \EuScript{P}}e_\frP \le \sum_{\frP \in \EuScript{P}} e_\frP \log_2|\OC/\frP| = \log_2|\OC /\frg|,
\]
which, in turn, is immediate from $\OC/\frg \cong \prod_{\frP \in \EuScript{P}} \OC/\frP^{e_\frP}$.
\end{proof}

\begin{remark}\label{rmk:start-with-f-or-g}
In the setup above, we start with a proper full two-sided ideal $\mathfrak{f}$ of $\mathcal{M}$ contained in 
$\Lambda$ and set $\mathfrak{g} := \mathfrak{f} \cap C$.
Under hypothesis (H) on $A$, we may instead start with a proper full ideal $\mathfrak{g}$ of 
$\mathcal{O}_{C}$ such that $\mathfrak{g}\mathcal{M}$ is contained in $\Lambda$ and then set
$\mathfrak{f} := \mathfrak{g}\calM$. 
In this situation, we then have $\frg = \frf \cap \OC$ by \cite[(27.6)]{Reiner2003}.
\end{remark}

\section{Lifting units of reduced norm one}

Let $K$ be a number field with ring of integers $\calO = \OK$.
Let $r \in \Z_{>0}$ and let $A = \prod_{i=1}^{r} \Mat_{n_i}(K_i)$
where $K_{i}$ is a finite field extension of $K$ and $n_{i} \in \Z_{>0}$ for each $i$.
In particular, $A$ is a finite-dimensional semisimple $K$-algebra satisfying hypothesis (H).
Let $C$ be the centre of $A$, which we can and do identify with $\prod_{i=1}^{r} K_{i}$.
In this situation, the reduced norm map $\nr \colon A \rightarrow C$
is equal to the product of maps $\det : \Mat_{n_i}(K_i) \rightarrow K_{i}$.

Let $\mathcal{M}$ be a maximal $\mathcal{O}$-order in $A$ and 
let $\mathcal{O}_{C} = \mathcal{M} \cap C = \prod_{i=1}^{r} \mathcal{O}_{K_{i}}$.
Then $\nr$ restricts to a group homomorphism
$\nr \colon \mathcal{M}^{\times} \rightarrow \mathcal{O}_{C}^{\times}$, which is surjective
since $A$ satisfies the Eichler condition relative to $\mathcal{O}$ 
(see \cite[(45.4), (45.6)]{curtisandreiner_vol2}).
Let $\mathfrak{g}$ be a proper full ideal of $\OC$
and let $\mathfrak{f} = \mathfrak{g}\mathcal{M}$.
Then by Lemma \ref{lem:defn-of-overline-nr} there exists a commutative diagram of groups
\begin{equation}\label{eq:comm-SL-nr-diagram}
\xymatrix@1@!0@=48pt { 
1 \ar[r] & \SL(\mathcal{M}) \ar[rr] \ar[d]^{f_{1}} & &  \mathcal{M}^{\times} \ar[rr]^{\nr}  
\ar[d]_{f} & & \mathcal{O}_{C}^{\times} \ar[r] \ar[d]_{f_{2}} &  1 \\
1 \ar[r] & \SL(\mathcal{M}/\mathfrak{f}) \ar[rr] & & (\mathcal{M}/\mathfrak{f})^{\times} 
\ar[rr]^{\overline{\nr}}  
& & ( \mathcal{O}_{C}/\mathfrak{g})^{\times} \ \ar[r] & 1,
} 
\end{equation}
where the rows are exact, $\SL(\mathcal{M})$ and $\SL(\mathcal{M}/\mathfrak{f})$ are defined by the exactness of these rows, and the vertical maps are induced by the canonical projections.
Note that this is consistent with diagram \eqref{f1 diagram}, but we do not require an order $\Lambda$
for the above setup.

The aim of this section is to show that, under the above assumptions on $A$ and $\mathfrak{f}$, the map
$f_{1}$ is surjective, and that there exists a polynomial-time
algorithm that given an element of $\SL(\calM/\frf)=\SL(\calM/\frg\calM)$ returns a preimage under $f_{1}$.

\subsection{Lifting unimodular matrices}\label{subsec:lifting-unimodular-matrices}
We first consider the case where $A = \Mat_{n}(K)$ and $\calM = \Mat_{n}(\calO)$
for some $n \in \Z_{>0}$.
In this situation, we have $\mathcal{O}=\mathcal{O}_{C}$ and 
$\mathcal{M}/\mathfrak{f} = \Mat_{n}(\calO/\mathfrak{g})$.
Moreover, both the maps $\nr$ and $\overline{\nr}$ in \eqref{eq:comm-SL-nr-diagram} 
are just the usual determinant maps, $\SL(\mathcal{M}) = \SL_{n}(\mathcal{O})$, and 
$\SL(\mathcal{M}/\mathfrak{f}) = \SL_{n}(\mathcal{O}/\mathfrak{g})$. 
Thus $f_{1}$ is the canonical map
$f_{1} : \SL_{n}(\mathcal{O}) \rightarrow \SL_{n}(\mathcal{O}/\mathfrak{g})$.
Note that this map is trivial when $n=1$, so we henceforth suppose that $n \geq 2$.

We use the following notation for a commutative ring $R$.
Given $1 \leq i, j \leq n$ with $i \neq j$, and $r \in R$ we denote by $e_{ij}(r) \in \SL_{n}(R)$
the matrix with ones on the diagonal and entry $r$ at position $(i, j)$.
We refer to these matrices as \textit{elementary matrices}.
Let $\E_{n}(R)$ denote the subgroup of $\SL_{n}(R)$ generated by all elementary matrices.

Since $\calO/\frg$ is semilocal, $\SL_{n}(\calO/\frg) = \E_{n}(\calO/\frg)$
by \cite[Chapter V, Corollary 9.2]{bass}.
Thus every element of $\SL_{n}(\calO/\frg)$ can be expressed as a product of elementary matrices,
and every such matrix can easily be lifted to an elementary matrix in $\SL_{n}(\calO)$.
This immediately implies the theoretical part of Corollary~\ref{cor:liftn} below.
However, we will need a constructive proof that then translates into an efficient algorithm.

We will show that, given the factorisation of $\mathfrak{g}$, there exists a polynomial time algorithm for lifting unimodular matrices over $\mathcal{O}/\frg$ to $\calO$.
The idea is to reduce to the local case and then apply the Chinese remainder theorem.

For any matrix $M$, let $M^{t}$ denote its transpose.
A vector $\mathbf{v} = (v_{1}, \ldots, v_{n})^{t}$ of elements of a commutative ring $R$ is said to 
be \emph{unimodular} if $\sum_{i=1}^{n} R v_{i} = R$.
In the following we denote by $\mathfrak{q} = \mathfrak{p}^l$, $l \in \Z_{>0}$, the power of a non-zero prime ideal of $\mathcal{O}$. Note that $\mathcal{O}/\mathfrak{q}$ is a local ring.

\begin{lemma}\label{lem:row}
There exists a polynomial-time algorithm that given a unimodular vector 
$\mathbf{v} \in (\calO/\mathfrak{q})^{n}$ returns elementary matrices $E_{1},\dotsc,E_{k} \in \Mat_{n}(\calO/\mathfrak{q})$
such that $E_{1} \dotsm E_{k} \mathbf{v} = (x,0,\dotsc,0)^{t}$ for some $x \in (\calO/\mathfrak{q})^{\times}$.
\end{lemma}

\begin{proof}
Write $\mathbf{v} = (v_1,\dotsc,v_n)^t$.
Note that as $\mathcal{O}/\mathfrak{q}$ is local, $\mathbf{v}$ being unimodular implies that there exists $1 \leq i \leq n$ such that $v_i \in (\mathcal{O}/\mathfrak{q})^\times$.
  
Case 1: 
If $v_1 \in (\calO/\mathfrak{q})^\times$, then $e_{21}(-v_1^{-1}v_2) \dotsm e_{n1}(-v_1^{-1}v_n) \mathbf{v}$
has the required form.

Case 2:
If $v_i \in (\calO/\mathfrak{q})^\times$ with $1 < i \leq n$, then after multiplying $\mathbf{v}$ by
$e_{1i}(1) e_{i1}(-1) e_{1i}(1)$ on the left, the first entry will be invertible and we are in the first case.
\end{proof}

\begin{lemma}\label{lem:upper}
There exists a polynomial-time algorithm that given a matrix 
$V \in \SL_{n}(\calO/\mathfrak{q})$ returns elementary matrices
$E_{1},\dotsc,E_{k} \in \Mat_{n}(\calO/\mathfrak{q})$ such that $E_{1} \dotsm E_{k} V$ is upper triangular.
If $V$ is lower triangular, then $E_{1} \dotsm E_{k} V$ is diagonal.
\end{lemma}

\begin{proof}
The first part follows by repeatedly applying Lemma~\ref{lem:row} to $V$ and submatrices of $V$. If $V$ is lower triangular, then we are always in Case 1 of the proof of Lemma~\ref{lem:row} and thus easily see that the resulting matrix is diagonal.
\end{proof}

As we will see below, the previous results allow us to transform 
unimodular matrices into diagonal matrices.
Thus it remains to consider unimodular diagonal matrices.

\begin{lemma}\label{lem:diag}
There exists a polynomial-time algorithm that given
  $V = \operatorname{diag}(v_1,\dotsc,v_n) \in \Mat_{n}(\calO/\mathfrak{q})$ with $\prod_{1 \leq i \leq n} v_i = 1$ returns elementary matrices $E_{1}, \dotsc, E_{k} \in \Mat_{n}(\calO/\mathfrak{q})$ such that $E_1 \dotsm E_k V$ is the $n \times n$ identity matrix.
\end{lemma}

\begin{proof}
From \cite[2.1.3 Corollary]{Rosenberg1994} it follows that a diagonal matrix with diagonal
  $(1,\dotsc,1,v,v^{-1},1,\dotsc,1)$ with $v \in (\calO/\mathfrak{q})^\times$ is the product of six elementary matrices, since
\[ 
\begin{pmatrix} v & 0 \\ 0 & v^{-1} \end{pmatrix} =
\begin{pmatrix} 1 & v \\ 0 & 1 \end{pmatrix}
\begin{pmatrix} 1 & 0 \\ -v^{-1} & 1 \end{pmatrix}
\begin{pmatrix} 1 & v \\ 0 & 1 \end{pmatrix}
\begin{pmatrix} 1 & -1 \\ 0 & 1 \end{pmatrix}
\begin{pmatrix} 1 & 0 \\ 1 & 1 \end{pmatrix}
\begin{pmatrix} 1 & -1 \\ 0 & 1 \end{pmatrix}.
\] 
Hence we can left-multiply $V$ with $6(n - 1)$ elementary matrices to obtain $(1,1,\dotsc,1)^t$.
\end{proof}

\begin{prop}\label{prop:lifting_help}
There exists a polynomial-time algorithm that given a matrix $V \in \SL_n(\calO/\mathfrak{q})$
returns elementary matrices $E_1,\dotsc,E_k, F_1,\dotsc,F_l \in \Mat_n(\calO/\mathfrak{q})$
such that $E_1\dotsm E_k V F_1\dotsc F_l$ is the $n \times n$ identity matrix.
\end{prop}

\begin{proof}
Using Lemma~\ref{lem:upper} there exist elementary matrices $E_1, \dotsc,E_{k'}$ such that 
$U := E_1\dotsm E_{k'} V$ is an upper triangular matrix.
Using Lemma~\ref{lem:upper} again, this time applied to the lower diagonal matrix $U^t$,
we can find elementary matrices $F_1,\dotsc,F_l$ such that $D = UF_1\dotsm F_l$ is a diagonal matrix.
Finally, invoking Lemma~\ref{lem:diag} yields elementary matrices
$\tilde E_1,\dotsc,\tilde E_{\tilde k}$ such that $\tilde E_1\dotsm \tilde E_{\tilde k} D$ is the $n \times n$ identity matrix.
\end{proof}

\begin{corollary}\label{cor:prodele}
  There exists a polynomial-time algorithm that given the factorisation of $\frg$ and a matrix
  $V \in \SL_{n}(\calO/\frg)$ returns elementary matrices $E_1,\dotsc,E_k \in \Mat_{n}(\calO/\frg)$ such that $V = E_1\dotsm E_k$.
\end{corollary}

\begin{proof}
In the case that $\frg$ is a prime ideal power, this follows from Proposition~\ref{prop:lifting_help}.
Now let $\mathfrak{g} = \mathfrak{q_1}\dotsm\mathfrak{q_m}$ be the product of $m$ coprime prime ideal powers 
and consider a matrix $V \in \SL_n(\calO/\frg)$. For each $1 \leq i \leq m$ we can determine in polynomial time a
factorisation of $V \in \SL_n(\calO/\mathfrak{q}_i)$ into elementary matrices.
The result follows by observing that the canonical map 
\[
  \E_n(\calO/\frg) \to \prod_{i=1}^m \E_n(\calO/\mathfrak{q}_i)
\]
is an isomorphism by the Chinese remainder theorem which can be made effective in
polynomial time~(\cite[Proposition~1.3.11]{MR1728313}).
\end{proof}

Since we can trivially lift elementary matrices along the canonical map
$f_{1}: \SL_{n}(\calO) \to \SL_{n}(\calO/\mathfrak{\frg})$, the same is true for arbitrary matrices in $\SL_{n}(\calO/\frg)$.

\begin{corollary}\label{cor:liftn}
There exists a polynomial-time algorithm that given the factorisation of $\mathfrak{g}$ and a matrix $V \in \SL_n(\calO/\frg)$ returns
$U \in \SL_n(\calO)$ such that $f_{1}(U) = V$.
\end{corollary}

\subsection{Lifting norm one units for nice maximal orders}
We now consider the case in which $A = \Mat_{n}(K)$ and 
$\calM = \calM_{\mathfrak a,n}$ is a nice maximal order as defined in \S \ref{subsec:nice},
where $\mathfrak{a}$ is a nonzero fractional ideal of $\mathcal{O}$ and $n \in \Z_{\geq 2}$.
(As in \S \ref{subsec:lifting-unimodular-matrices}, the case $n=1$ is trivial.)
Some of the ideas used here are based on \cite[\S 6]{BJ08}.

Let $\mathfrak{b}$ be an integral ideal of $\mathcal{O}$ such that $\mathfrak{b} + \frg = \calO$
and $\mathfrak{a} = \xi \mathfrak{b}$ for some $\xi \in K^{\times}$.
Such an ideal $\mathfrak{b}$ and element $\xi$ can be computed in probabilistic polynomial time,
as shown in Corollary~\ref{cor:coprime}.
Let $b \in \mathfrak b, y \in \frg$ such that $b + y = 1$ and let $R=\mathcal{O}/\mathfrak{g}$.
Then we have an isomorphism
$\mathcal{O}/\mathfrak{g} \rightarrow \mathfrak{a}/\mathfrak{a}\mathfrak{g}$
of $R$-modules defined by
$z + \mathfrak{g} \mapsto zb\xi + \mathfrak{a}\mathfrak{g}$,
with the inverse map given by $x + \mathfrak{a}\mathfrak{g} \mapsto \xi^{-1}x + \mathfrak{g}$.
The first of these maps induces an isomorphism
$\theta_{1} : R^{\oplus n}
\rightarrow R^{\oplus n-1}
\oplus \mathfrak{a}/\mathfrak{a}\mathfrak{g}$ of $R$-modules,
and the second map induces an inverse $\theta_{2}$.
Define $n \times n$ diagonal matrices
$\Phi_{1} = \operatorname{diag}(1, \ldots, 1, \xi^{-1})$ and $\Phi_{2} = \operatorname{diag}(1, \ldots, 1, b\xi)$.
Then we have maps
\[
\psi_1 \colon \Mat_n(\calO) \lra \calM, \, X \mapsto \Phi_2 X \Phi_1
\quad \text{ and } \quad
\psi_2 \colon \calM \lra \Mat_n(\calO), \, Y \mapsto \Phi_1 Y \Phi_2.
\] 
These maps are not multiplicative in general.
However, since $\theta_{1}$ and $\theta_{2}$ are mutually inverse isomorphisms, 
we see that
$\psi_{1}$ and $\psi_{2}$ induce mutually inverse isomorphisms
\[
\overline{\psi}_1 \colon \GL_n(\calO/\frg) \to (\calM/\frg\calM)^\times
\quad \text{ and } \quad
\overline{\psi}_2 \colon (\calM/\frg\calM)^\times \to \GL_n(\calO/\mathfrak{g}).
\]

\begin{lemma}\label{lem:twistlift}
Let $\overline{E} \in \SL_n(\calO/\frg)$ be an elementary matrix.
Then $\overline{\psi}_{1}(\overline{E})$ can be lifted to an element
$U \in \calM^\times$ with $\nr(U) = 1$.
\end{lemma}

\begin{proof}
For $V \in \Mat_n(\calO)$ we write 
\[
  \def\arraystretch{1.2}
V =   \left(  \begin{array}{c|c}V_1 & x \\ \hline y & d \end{array} \right)
\]
with $V_1 \in \Mat_{n-1}(\calO), x,y^t \in \calO^{n-1}$ and $d \in \calO$. Then
\[
  \def\arraystretch{1.2}
\psi_{1}(V) = \left(  \begin{array}{c|c}V_1 & \xi^{-1}x \\ \hline \xi b y & bd \end{array} \right).
\]
Let $I_m \in \Mat_m(\calO)$ denote the identity matrix.

\noindent
Case 1: If $\def\arraystretch{1.2}
\overline{E} =  \left(  \begin{array}{c|c} e_{ij}(\bar a) & 0 \\ \hline 0 & 1 \end{array} \right)
$
with $a \in \calO$, then 
$\def\arraystretch{1.2}
\overline{\psi}_1(\overline{E}) =  \left(  \begin{array}{c|c} e_{ij}(\bar a) & 0 \\ \hline 0 & \bar b \end{array} \right)
$
and a lift is given by 
\[
\left(  \begin{array}{c|c} e_{ij}(a) & 0 \\ \hline 0 & 1 \end{array} \right).
\]

\noindent
Case 2: If $\def\arraystretch{1.2}
\overline{E} =  \left(  \begin{array}{c|c} \overline{I}_{n-1} & \bar x \\ \hline 0 & 1 \end{array} \right)
$
with $\bar x^t = (0,\ldots,0,\bar a, 0, \ldots, 0), a \in \calO$,
then 
$\def\arraystretch{1.2}
\overline{\psi}_{1}(\overline{E}) 
= \left(  \begin{array}{c|c} \overline{I}_{n-1} &  \overline{\xi^{-1} x}\\ \hline 0 & \bar b \end{array} \right)
$
and a lift is given by
\[
\left(  \begin{array}{c|c} I_{n-1} &  {\xi^{-1} x}\\ \hline 0 & 1 \end{array} \right).
\]
Note that in this case $\xi^{-1}a \in \frb\fra^{-1} \sseq \fra^{-1}$.

\noindent
Case 3: If $\overline{E} =  \left(  \begin{array}{c|c} \overline{I}_{n-1} & 0 \\ \hline \bar y & 1 \end{array} \right)$
with $\bar y = (0,\ldots,0,\bar a, 0, \ldots, 0), a \in \calO$, then
$\def\arraystretch{1.2}
\psi_{1}(\overline{E}) = \left(  \begin{array}{c|c} \overline{I}_{n-1} &  0 \\ \hline \overline{\xi by} & \bar b \end{array} \right)$
and a lift is given by
\[
\left(  \begin{array}{c|c} I_{n-1} &  0 \\ \hline \xi by & 1 \end{array} \right).
\]
Here we note that $\xi ba  \in \xi\frb = \fra$.
\end{proof}

\begin{prop}\label{prop:liftnormone}
For $A = \Mat_{n}(K)$ let $\calM = \calM_{\mathfrak a,n} \sseq A$ be a nice maximal order.
Then there exists a probabilistic polynomial-time algorithm that given the factorisation of $\frg$ and $V \in \SL(\calM/\frg\calM)$ returns
$U \in \SL(\calM)$ with $f_{1}(U) = V$.
\end{prop}

\begin{proof}
By Corollary~\ref{cor:prodele} we can find elementary matrices 
$E_1,\dotsc,E_r \in \Mat_n(\calO/\frg)$ with $\overline{\psi}_2(V) = E_1\dotsm E_r$.
Applying $\overline{\psi}_{1}$ we obtain
$V = \overline{\psi}_1(E_1)\dotsm \overline{\psi}_1(E_r)$.
Moreover, by Lemma~\ref{lem:twistlift} each of the matrices $\overline{\psi}_1(E_i)$ can be lifted to
a matrix $U_i \in \calM^\times$ with $\nr(U_{i}) = 1$.
Thus we can and do take $U := \prod_{i}U_{i}$.
\end{proof}

\subsection{Lifting norm one units in maximal orders}

We now consider an arbitrary maximal order $\calM$ of $A = \prod_{i=1}^{r}\Mat_{n_{i}}(K_{i})$.

\begin{theorem}\label{coro:normonelift}
The map $f_{1} \colon \SL(\calM) \to \SL(\calM/\frg \calM)$ is surjective.
Moreover, there exists a probabilistic polynomial-time algorithm that given
the factorisation of $\frg$ and $V \in \SL(\calM/\frg \calM)$ returns an element $U \in \SL(\calM)$ with $f_{1}(U) = V$.
\end{theorem}

\begin{proof}
By decomposing $\calM$ using the central primitive idempotents, it is sufficient to consider the case
$A = \Mat_{n}(K)$.
By Lemma~\ref{lem:nice}, we can and do assume that $\calM = \calM_{\mathfrak a, n}$
is a nice maximal order.
Thus the result follows from Proposition~\ref{prop:liftnormone}.
\end{proof}

\section{Isomorphism testing and the principal ideal problem}\label{sec:pip}

Let $K$ be a number field with ring of integers $\calO = \OK$ and let $A$ be a finite-dimensional
$K$-algebra satisfying hypothesis (H).
Let $\Lambda$ be an $\mathcal{O}$-order in $A$. 
In this section, we present the main algorithm for solving the isomorphism problem $\pisomorphic$ for lattices
over $\Lambda$.

We begin with two straightforward reductions which together show that it suffices 
to consider the problem $\ppip$ in the case that $A$ is semisimple.
Note that these reductions are valid when $A$ is an arbitrary finite-dimensional $K$-algebra that does not necessarily satisfy hypothesis (H).

\begin{prop}\label{prop:comp_red_rankone}
The problem $\pisomorphic$ is polynomial-time reducible to $\ppip$.
More precisely, for $\Lambda$-lattices $X$ and $Y$, the problem $\pisomorphic$ is polynomial-time 
reducible to $\ppip$ for an $\End_\Lambda(Y)$-lattice in $\End_A(KY)$.
\end{prop}

\begin{proof}
Let $X$ and $Y$ be two $\Lambda$-lattices.
By~\cite[Corollary 3]{MR1809971} we can check in polynomial time whether the $A$-modules $KX$ and $KY$ are isomorphic,
and if so, compute an isomorphism $f \colon KY \to KX$.
Then $\Phi \colon \Hom_A(KX, KY) \to \End_A(KY), \, g \mapsto g \circ f$ is an 
isomorphism of $\End_A(KY)$-modules.
Recall from \S \ref{subsec:reduce-free-rank1} 
that we consider $\Hom_\Lambda(X, Y)$ as a subset of $\Hom_A(KX, KY)$.
Thus by Proposition~\ref{prop:hom-free-over-end}, the $\Lambda$-lattices $X$ and $Y$ are isomorphic
if and only if the full $\End_\Lambda(Y)$-lattice $\Phi(\Hom_\Lambda(X, Y))$ in $\End_A(KY)$ is
free of rank $1$ and for every (any) free generator $\alpha$ the morphism $\alpha \circ f^{-1} \colon X \to Y$ 
is an isomorphism.
\end{proof}

Let $\jac(A)$ denote the Jacobson radical of $A$ and recall that
$\overline{A} := A/{\jac(A)}$ is a semisimple $K$-algebra by \cite[(5.19)]{curtisandreiner_vol1}.
For any full $\Lambda$-lattice $X$ in $A$, let $\overline{X}$ denote its image under the canonical projection map $A \rightarrow \overline{A}$. Note that $\overline{\Lambda}$ is an $\mathcal{O}$-order in $\overline{A}$.

\begin{prop}\label{prop:comp_red_semisimple}
The problem $\ppip$ for an arbitrary finite-dimensional $K$-algebra
is polynomial-time reducible to $\ppip$ for a finite-dimensional semisimple $K$-algebra.
More precisely, for a full $\Lambda$-lattice $X$ in $A$, the problem $\ppip$ is polynomial-time reducible to the problem $\ppip$ for the full $\overline{\Lambda}$-lattice $\overline{X}$ in $\overline{A}$.
\end{prop}

\begin{proof}
The Jacobson radical of $A$ can be computed in polynomial time by~\cite[1.5~A]{Friedl1985}.
The result then follows from Theorem~\ref{thm:lift-gen-from-semisimple-case}.
\end{proof}

The main algorithm of the present article is as follows.

\begin{algorithm}\label{alg:main} 
Suppose that $A$ is semisimple and satisfies hypothesis \textup{(H)}. 
Let $X$ be a full $\Lambda$-lattice in $A$.
The following steps solve $\ppip(X)$, that is, they 
determine whether there exists $\alpha \in X$ such that $X=\Lambda \alpha$ 
and if so, return such an element $\alpha$.
\renewcommand{\labelenumi}{(\arabic{enumi})}
\begin{enumerate}
\item Determine  the centre $C$ of $A$, the decomposition $A = \prod_{i} A_i$ into simple $K$-algebras $A_i$
and, for each $i$,  an isomorphism $A_i \cong \Mat_{n_i}(K_i)$. 
\item Compute a maximal $\mathcal{O}$-order $\mathcal{M}$ in $A$ containing $\Lambda$
and its centre $\mathcal{O}_{C} := \mathcal{M} \cap C$.
\item Compute the central primitive idempotents $e_{i}$ and the components 
$\mathcal{M}_{i}:=\mathcal{M} e_{i}$.
\item Compute the central conductor $\mathfrak{g} := \{ x \in C \mid x\mathcal{M} \subseteq \Lambda \}$ 
of $\Lambda$ in $\mathcal{M}$ and $\mathfrak{f} := \mathfrak{g}\mathcal{M}$.
\item Check whether $\mathcal{M}X$ is free over $\mathcal{M}$, and if so, compute
$\beta$ such that $\calM X = \calM \beta$.
\item Check whether $X$ is locally free over $\Lambda$.
\item Replace $X$ by $X \xi$, where $\xi \in A^\times$ is such that $X\xi \subseteq \Lambda$ and 
$X\xi + \frf = \Lambda$.
\item Compute a set of generators for $\left( \Lambda / \frf\right)^\times$.
\item Let $\overline\nr \colon \left( \mathcal{M} / \frf\right)^\times \lra \left( \OC/\frg \right)^{\times}$
be the map of Lemma \ref{lem:defn-of-overline-nr}. 
Compute $\left( \OC/\frg \right)^\times$ as an abstract abelian group and compute
$\overline\nr\left( (\Lambda / \frf)^\times\right)$ as a subgroup of $\left( \OC/\frg \right)^\times$.
\item  Let $\pi_{2}$ and $f_2$ be the maps  defined in the commutative diagram \eqref{f1 diagram}.
Decide whether $\overline\nr(\overline\beta)$ is in the image of $\pi_2\circ f_2$, and if so,
compute $\bar a \in \left( \Lambda / \frf\right)^\times$  and $u \in \calM^\times$
such that $\overline\nr(\overline{\beta a}) = \overline\nr(\bar u)$. 
\item Compute $v \in \SL(\calM)$ such that $\overline{\beta a u^{-1}} = \overline v$.
\end{enumerate}
If any of steps \textup{(5)}, \textup{(6)}, or \textup{(10)} fail, then $X$ is not free over $\Lambda$.
If all these steps succeed, then $X = \Lambda \alpha$ where $\alpha := (vu)^{-1}\beta$.
\end{algorithm}

\begin{proof}[Proof of correctness of Algorithm \ref{alg:main}]
Failure of Step (5) or (6) immediately implies
that $X$ is not free over $\Lambda$. 
Otherwise, we use the local bases computed in Step (6) to replace
$X$ by $X\xi$ in Step (7), as described in Proposition \ref{prop:suitable-locally-free-lattice} and its proof. After successful completion of Steps (1) to (7), we then can and do assume that $X$ is a locally free full 
$\Lambda$-lattice in $A$ such that $X+\frf = \Lambda$ and $\calM X = \calM\beta$. 
These are the assumptions needed for Theorem \ref{thm:main-thm-for-freeness}.  
Moreover, since $A$ is semisimple and satisfies hypothesis (H), 
the map $f_1$ in diagram \eqref{f1 diagram} is surjective by Theorem \ref{coro:normonelift},
and so Theorem \ref{thm:main-thm-for-freeness} (b) can be applied.
Hence
$X$ is free over $\Lambda$ if and only if  $\overline\nr(\overline\beta)$ is contained in the image of $\pi_2\circ f_2$. 
This is precisely what is checked in Step (10). 
In addition, the second part of Theorem \ref{thm:main-thm-for-freeness} (b) implies
that $X = \Lambda\alpha$ with $\alpha$ as at the end of Algorithm \ref{alg:main}.
\end{proof}

The following result analyses the complexity of Algorithm \ref{alg:main},
and further details on each step are given in the proof.

\begin{theorem}\label{thm:main}
Let $\Lambda$ be an $\mathcal{O}$-order in a finite-dimensional semisimple $K$-algebra $A$ satisfying hypothesis \textup{(H)} and let $K_{1}, \ldots, K_{r}$ be the simple components of the centre of $A$. 
Let $\mathcal{M}$ be any choice of maximal $\mathcal{O}$-order in $A$ containing $\Lambda$
and let $\mathfrak{h} = [\mathcal{M} : \Lambda]_{\mathcal{O}}$ be the module 
index of $\Lambda$ in $\mathcal{M}$.
Then for a full $\Lambda$-lattice $X$ in $A$, Algorithm~\ref{alg:main} reduces the problem $\ppip(X)$ in probabilistic polynomial time to
\begin{enumerate}
\item $\pwedderburn(A)$, the computation of the Wedderburn decomposition of $A$,
\item $\pfactor(\disc(\Lambda))$, the factorisation of the discriminant of $\Lambda$,
\item for each $i$ with $1 \leq i \leq r$, one instance of $\ppip_{\mathcal O_{K_i}}$,
\item for each $i$ with $1 \leq i \leq r$, $\unit(\mathcal O_{K_i})$,
\item for each prime ideal divisor $\mathfrak{p}$ of $\mathfrak{h}$, the problem $\pdlog$ 
for extensions of $\mathcal{O}/\mathfrak{p}$, and
\item for each prime ideal divisor $\mathfrak{p}$ of $\mathfrak{h}$,
the problem $\pprimitive$ for extensions of $\mathcal{O}/\mathfrak{p}$.
\end{enumerate}
Note that $\mathcal{M}$ and $\mathfrak{h}$ are not part of the input and $\mathfrak{h}$
is only needed for the above complexity statement. 
Moreover, $\mathfrak{h}$ does not depend on the choice of $\mathcal{M}$.
\end{theorem}

\begin{proof}
In the following, the steps refer to those of Algorithm \ref{alg:main}.
Let $\calM$ be the maximal order computed in Step~(2) and let $\frf$ be the ideal computed in Step~(4).
Before analysing the steps, we make the following observations. 
By \cite[(25.3)]{Reiner2003}, $\disc(\mathcal{M})$ is independent of the choice of $\mathcal{M}$.
Moreover, by \cite[(26.3)(iii)]{curtisandreiner_vol1} we have
$\disc(\Lambda) = \mathfrak{h}^{2} \disc(\calM)$, and so $\disc(\Lambda)$ and $\mathfrak{h}$ are also
independent of the choice of $\mathcal{M}$.
Since $\frh\calM \sseq \Lambda$, we have $\frh \sseq \frg$ for any choice of $\mathcal{M}$.
Therefore
$\disc(\Lambda) \subseteq \mathfrak{h}\subseteq \mathfrak{g}$ and $\frh\Lambda \subseteq \frf$.
In particular, $\pfactor(\calO \cap \frf)$ and $\pfactor(\mathfrak{g})$ reduce in polynomial time
to $\pfactor(\disc(\Lambda))$.

Step~(1) is an instance of $\pwedderburn$.
In Step~(2), the problem of computing a maximal $\mathcal{O}$-order $\mathcal{M}$ in $A$
containing $\Lambda$ reduces in probabilistic polynomial time to $\pfactor(\disc(\Lambda))$
by Proposition \ref{prop:comp-max-order}.
It is then trivial to determine $\OC = \calM \cap C$.
Step~(3) can be easily performed using the isomorphisms $A_i \cong \Mat_{n_i}(K_i)$ from Step~(1).
In Step~(4), the central conductor $\mathfrak{g}$ can be computed as the intersection
$(\mathcal{M} : \Lambda)_l \cap C$, where
$(\mathcal{M} : \Lambda)_l := \{ x \in \mathcal{M} \mid x\mathcal{M} \subseteq \Lambda \}$ is
 the left conductor of $\Lambda$ into $\mathcal{M}$.
 As the left conductor can be determined using a pseudo-Hermite normal form computation
(see~\cite[(2.16)]{Friedrichs}), this step can also be performed in polynomial time. 
  Step~(5) is probabilistic polynomial-time reducible to one instance of $\ppip_{\mathcal O_{K_i}}$
for each $1 \leq i \leq r$, by Proposition~\ref{prop:pipmaximal}.
Steps (6) and (7) are probabilistic polynomial-time reducible to $\pfactor(\calO \cap \frf)$ by
Proposition~\ref{prop:suitable-locally-free-lattice}.

Step (8): Proposition~\ref{prop:compk1} shows that this is probabilistic polynomial-time reducible to
$\pfactor(\mathfrak{g})$ and
{for each prime ideal divisor $\mathfrak{P}$ of $\mathfrak{g}$ at most
$d$ instances of $\pprimitive$ in extensions of $\mathcal{O}/(\mathcal{O}\cap\mathfrak{P})$ of degree at most $d$.
Now for each prime ideal $\mathfrak{p}$ dividing $\mathfrak{g} \cap \mathcal{O}$ there are at most $d$
prime ideals $\mathfrak{P}$ of $\mathcal{O}_C$ satisfying $\mathfrak{P} \cap \mathcal{O} = \mathfrak{p}$.
Finally note that $\mathfrak{g} \cap \mathcal{O}$ divides $\mathfrak{h}$.}

Step (9): It follows from \cite[Algorithms 4.2.2 and 4.2.17]{MR1728313} 
that the computation of generators and the structure of $\left( \OC/\frg \right)^\times$ as an abelian group
is polynomial-time reducible to $\pfactor(\mathfrak{g})$
and for each prime ideal divisor $\mathfrak{P}$ of $\mathfrak{g}$ one instance of $\pprimitive$ in an extension $\mathcal{O}/(\mathcal{O}\cap\mathfrak{P})$ of degree at most $d$.
Estimating the number of prime ideal divisors as in the previous paragraph shows that this part contributes $d$ instances of $\pprimitive$ in (e).
Let $V = \{\bar a_1, \ldots, \bar a_m\}$ be a set of generators of $\left( \Lambda / \frf \right)^\times$.
Let $e$ denote the exponent of $\left( \OC / \frg \right)^\times$ and let 
$\calG := \prod_{i=1}^m \Z / e\Z \cdot\bar a_i$ be the $\Z/e\Z$-free abelian group on $V$.
Let $\overline\nu \colon \calG \rightarrow \left( \OC / \frg \right)^\times$ be the homomorphism induced
by $\bar a_i \mapsto \overline\nr(\bar a_i)$.
Then $\im(\bar\nu) = \overline\nr\left( (\Lambda / \frf)^\times\right)$  and we apply algorithms for finite abelian groups
(see \cite[\S 4.1]{MR1728313}) to compute the image.
For this we have to solve the discrete logarithm in $(\OC/\frg)^\times$ for each of the $m$ generators
$\bar a_1, \ldots, \bar a_m$. By Proposition \ref{prop:compk1} the number $m$ is bounded by $5d[K:\Q]\log_2|\OC/\frg|$. 
Thus the claim in part (f) follows from
\begin{align*}
|\OC/\frg| 
& \leq |\OC/\frh\OC| 
= \prod_{i=1}^{r} |\calO_{K_i}/\frh\calO_{K_i}| 
= \prod_{i=1}^{r} \Nm_{K_i/\Q}(\frh\calO_{K_i}) 
= \prod_{i=1}^{r} \Nm_{K/\Q}(\frh)^{[K_i:K]} \\
&=  \Nm_{K/\Q}(\frh)^{[C:K]} 
\leq \Nm_{K/\Q}(\frh)^d. 
\end{align*}
Note that solving the discrete logarithm in $(\OC/\frg)^\times$ requires solving the discrete logarithm problem in $(\OC/\mathfrak{P})^\times$ for all prime ideals $\mathfrak{P}$ dividing $\frg$.
As in Step~(8), for each prime ideal $\mathfrak{p}$ dividing $\mathfrak{g}\cap \mathcal{O}$ there are at most $d$ prime ideals $\mathfrak{P}$ of $\OC$ with $\mathfrak{P} \cap \mathcal{O} = \mathfrak{p}$
and for each of those prime ideals $\OC/\mathfrak{P}$ is an extension of $\mathcal{O}/\mathfrak{p}$ of degree at most $d$.

Step (10): The reduced norm map $\nr \colon \calM^\times \rightarrow \OC^\times$ is surjective
by Lemma~\ref{lem:normeq}.
The computation of $\OC^\times$ is performed componentwise and thus reduces 
to $\unit(\mathcal O_{K_i})$ for $1 \leq i \leq r$. 
We then determine the image of the canonical projection $\OC^\times \rightarrow \left( \OC/\frg \right)^\times / \im(\bar\nu)$
as an abstract subgroup of $ \left( \OC/\frg \right)^\times / \im(\bar \nu)$. 
This again requires an instance of solving the discrete logarithm in
$\left( \OC/\frg \right)^\times$ for each of the generators of $\OC^\times$. 
By Dirichlet's unit theorem the number of these generators can be bounded by $d$.
Applying standard algorithms for finite abelian groups it is then straightforward to decide whether
$\overline\nr(\bar \beta)$ is contained in the image of
$\OC^\times \rightarrow  {\left( \OC/\frg \right)^\times} /{\im(\bar \nu)}$, and
if so, to compute $\epsilon \in \OC^\times$, $a \in \left(\Lambda / \frf\right)^\times$
such that $\bar\epsilon \equiv \overline{\beta a} \pmod{\frg}$.
An element $u \in \calM^\times$ such that $\nr(u) = \epsilon$ can be found using Lemma~\ref{lem:normeq} in probabilistic polynomial time.
Note that this step requires one more instance of solving the discrete logarithm in $(\OC/\frg)^{\times}$,
which was already analysed in Step~(9).

Step~(11): As the factorisation of $\frg$ is known, this can be done in probabilistic polynomial time by Theorem~\ref{coro:normonelift}.
\end{proof}

We now consider the case in which we allow certain pre-computations that only depend on the order 
$\Lambda$ and not on the $\Lambda$-lattice $X$.

\begin{corollary}
Fix an $\mathcal{O}$-order $\Lambda$ in a finite-dimensional semisimple $K$-algebra $A$ satisfying hypothesis \textup{(H)} and let $K_{1}, \ldots, K_{r}$ be the simple components of the centre of $A$. 
Then for a full $\Lambda$-lattice $X$ in $A$, the problem $\ppip(X)$ reduces in probabilistic polynomial to
  $\ppip$ for $\mathcal{O}_{K_i}$, $1 \leq i \leq r$, and $\pdlog$.
\end{corollary}

\begin{proof}
In Algorithm~\ref{alg:main} we may consider all steps which do not depend on $X$ as precomputations. Then for each lattice $X$ only Steps~(5), (6), (7), (10) and (11) have to be performed.
The claim follows as in the proof of Theorem~\ref{thm:main}.
\end{proof}

\begin{remark}
In Theorem~\ref{thm:main} better results can be obtained by describing the complexity in terms of the central conductor $\mathfrak{g}$ (which depends not only on the order $\Lambda$, but also on the maximal order computed during the algorithm) instead of $\mathfrak{h}$.
More precisely, (e) and (f) can be replaced by
\begin{enumerate}
\item[(e$'$)] for each prime ideal divisor $\mathfrak{P}$ of $\mathfrak{g}$, the problem 
$\pdlog$ for extensions of $\mathcal{O}/ (\calO \cap \mathfrak{P})$,
\item[(f$'$)] for each prime ideal divisor $\mathfrak{P}$ of $\mathfrak{g}$, the problem 
$\pprimitive$ for extensions of $\calO / (\calO \cap \mathfrak{P})$.
\end{enumerate}
\end{remark}

\begin{remark}
By Remark~\ref{remark:need_splitting}, in Theorem~\ref{thm:main}, (a) can be replaced by
\begin{enumerate}
\item[(a$'$)]
$\psplitting(A_i)$ for $1 \leq i \leq r$, where $A = \bigoplus_{i=1}^r A_i$ is the decomposition into simple $K$-algebras.
\end{enumerate}
Note that we have formulated Theorem~\ref{thm:main} using $\pwedderburn$ since for certain families 
of algebras, one can directly solve $\pwedderburn$ in polynomial time (which would not necessarily be true after passing to the simple components).
This happens, for example, for certain algebras of the form $A/{\jac(A)}$ that appear
in the similarity problem for matrices over rings of integers of number fields
(see \S \ref{subsec:jacobson-comp}).
\end{remark}

\begin{remark}\label{remark:comp}
In view of Remarks~\ref{remark:subcomp} and~\ref{remark:need_splitting}, as well as the reductions of
Propositions \ref{prop:comp_red_rankone} and \ref{prop:comp_red_semisimple},
the problems $\ppip_{\Lambda}$ and $\pisomorphic_{\Lambda}$ for orders
$\Lambda$ in finite-dimensional $K$-algebras $A$ satisfying hypothesis (H) 
reduces 
\begin{enumerate}
\item in probabilistic subexponential time to \punit{} and \ppip{} for rings of integers of number fields, 
and $\pwedderburn$ or $\psplitting$, 
\item in quantum polynomial time to $\pwedderburn$ or $\psplitting$.
\end{enumerate} 
\end{remark}

\section{Application: similarity of matrices over rings of integers}\label{sec:conjugacy}

After proving some general results on the similarity of matrices over commutative rings,
we will give an application of Algorithm \ref{alg:main} to the similarity problem for matrices
over rings of integers of number fields.

\subsection{Similarity of matrices over commutative rings}
Let $R$ be a commutative ring and let $n \in \Z_{>0}$.
Recall that two matrices $A,B \in \Mat_{n}(R)$ are said to be \textit{similar over $R$} if
there exists a  \textit{conjugating matrix} $C \in \GL_{n}(R)$ such that $B=CAC^{-1}$.

We will adopt the setup of Faddeev~\cite{MR0194432}.
For $A,B \in \Mat_{n}(R)$ we define 
\[
C_R(A, B) = \{ X \in \Mat_n(R) \mid  XA = BX \}
\quad \text{ and } \quad
C_{R}(B) = C_{R}(B,B).
\]
Note that $C_{R}(B)$ is an $R$-algebra and that $C_{R}(A,B)$ is a (left) $C_{R}(B)$-module.

\begin{lemma}\label{lem:similar-matrices-gives-isomorphism}
Suppose there exists $C \in \GL_{n}(R)$ such that $B=CAC^{-1}$. 
Then the maps
\begin{align*}
\theta_{C} : C_R(B) \longrightarrow C_R(A,B), \quad & X \longmapsto XC,\\
\theta_{C^{-1}} : C_R(A, B) \longrightarrow C_R(B), \quad & X \longmapsto XC^{-1},
\end{align*}
are mutually inverse $C_{R}(B)$-module isomorphisms.
\end{lemma}

\begin{proof}
If $X \in C_{R}(B)$, then $B(XC)=XBC=X(CAC^{-1})C=(XC)A$ and so $XC \in C_{R}(A,B)$.
Hence the map $\theta_{C}$ is well defined. Similarly, the map $\theta_{C}^{-1}$ is also 
well defined and it is clear that $\theta_{C}$ and $\theta_{C^{-1}}$ are mutually inverse. 
\end{proof}

\begin{prop}\label{prop:similarity-gen-unit}
Two matrices $A,B \in \Mat_n(R)$ are similar over $R$ if and only if
\begin{enumerate}
\item the $C_R(B)$-module $C_R(A, B)$ is free of rank $1$, and
\item every (any) free generator $C$ of $C_R(A, B)$ over $C_{R}(B)$ is in $\GL_n(R)$.
\end{enumerate}
Furthermore, when this is the case, $C$ as in part \rm{(b)}
satisfies $B=CAC^{-1}$.
\end{prop}

\begin{proof}
Suppose that (a) and (b) hold and let $C$ be a free generator of $C_{R}(A, B)$ over $C_{R}(B)$.
In particular, $C \in C_{R}(A,B) \cap \GL_{n}(R)$ and it easily follows that $B=CAC^{-1}$. 
Suppose conversely that there exists $C \in \GL_{n}(R)$ such that $B=CAC^{-1}$. 
Then $\theta_{C}$ is an isomorphism by Lemma \ref{lem:similar-matrices-gives-isomorphism}
and so $C$ is a free generator of $C_{R}(A,B)$ over $C_{R}(B)$.
Thus (a) holds. Now let $D$ be any free generator of $C_{R}(A,B)$ over $C_{R}(B)$. 
Then there exists $E \in C_{R}(B)^{\times} \subseteq \GL_{n}(R)$ such that $D=EC$ and so 
$D \in \GL_{n}(R)$. Thus (b) holds. 
\end{proof}

The following result was proven by Faddeev \cite[Theorem 2]{MR0194432} in the case $R=\Z$,
though it was expressed in terms of ideals rather than modules. Moreover, 
Guralnick \cite[Theorem 6]{MR568794} observed that the proof works for any integral domain.
We include a short proof for the convenience of the reader and for comparison as per 
Remark \ref{rmk:comparision-similarity-results}.

\begin{prop}\label{prop:similarity-local-condition}
Suppose that $R$ is an integral domain.
Two matrices $A,B \in \Mat_n(R)$ are similar over $R$ if and only if
\begin{enumerate}
\item the $C_R(B)$-module $C_R(A, B)$ is free of rank $1$, and
\item for every maximal ideal $\mathfrak{p}$ of $R$, the matrices $A$ and $B$ are similar over $R_{\mathfrak{p}}$.
\end{enumerate}
Furthermore, when this is the case, 
any free generator $C$ of $C_R(A, B)$ over $C_{R}(B)$ satisfies $B=CAC^{-1}$.
\end{prop}

\begin{proof}
Suppose that $A,B \in \Mat_n(R)$ are similar over $R$. Then (b) clearly holds and (a) holds by 
Proposition \ref{prop:similarity-gen-unit}. Suppose conversely that (a) and (b) hold. 
Let $C$ be a free generator of $C_{R}(A,B)$ over $C_{R}(B)$.
Let $\mathfrak{p}$ be a maximal ideal of $R$. 
Then there exists $C_{\mathfrak{p}} \in \GL_{n}(R_{\mathfrak{p}})$ such that 
$B=C_{\mathfrak{p}}AC_{\mathfrak{p}}^{-1}$ and so $C_{\mathfrak{p}} \in C_{R_{\mathfrak{p}}}(A,B)$.
Since $C$ is also a free generator of $C_{R_{\mathfrak{p}}}(A,B)$ over $C_{R_{\mathfrak{p}}}(B)$,
there exists $D_{\mathfrak{p}} \in C_{R_{\mathfrak{p}}}(B)$ such that 
$C_{\mathfrak{p}} = D_{\mathfrak{p}} C$. 
Then $\det(D_{\mathfrak{p}})\det(C)=\det(C_{\mathfrak{p}}) \in R_{\mathfrak{p}}^{\times}$ and 
so $\det(C) \in R_{\mathfrak{p}}^{\times}$.
Moreover, by \cite[(4.2)(iv)]{curtisandreiner_vol1} we have $R = \cap_{\mathfrak{p}} R_{\mathfrak{p}}$
which implies that $R^{\times} = \cap_{\mathfrak{p}} R_{\mathfrak{p}}^{\times}$, 
where in both cases the intersection ranges over all maximal ideals $\mathfrak{p}$ of $R$.
Therefore $\det(C) \in R^{\times}$ and so $C \in \GL_{n}(R)$.
In particular, $C \in C_{R}(A,B) \cap \GL_{n}(R)$ and it easily follows that $B=CAC^{-1}$.
\end{proof}

\begin{remark}\label{rmk:comparision-similarity-results}
Propositions \ref{prop:similarity-gen-unit} and \ref{prop:similarity-local-condition} and their proofs
are analogues of Propositions \ref{prop:hom-free-over-end} and \ref{prop:hom-end-locally-iso}, respectively.
Indeed, in the case that $R$ is a noetherian integral domain, 
the former can be deduced from the latter, though it is easier to give more direct proofs of
more general results. 
Moreover, as well as having weaker hypotheses, Proposition \ref{prop:similarity-gen-unit} is better 
suited to algorithmic applications than Proposition \ref{prop:similarity-local-condition}.
\end{remark}

\subsection{The similarity problem in terms of modules over polynomial rings}
Let $R$ be a commutative ring and let $n \in \Z_{>0}$.
Let $R[x]$ be a polynomial ring in one variable over $R$.
For $A \in \Mat_{n}(R)$, we define $T_{R}(A)$ to be the $R[x]$-module $R^{n}$ with the action
$x v = Av$ for $v \in R^{n}$. 

\begin{lemma}\label{lem:homend}
Let $A,B,C \in \Mat_{n}(R)$. 
Define $\psi_{A,B,C} : T_{R}(A) \rightarrow T_{R}(B)$ by $v \mapsto Cv$.
Then $C \in C_{R}(A,B)$ if and only if $\psi_{A,B,C}$ is an $R[x]$-module homomorphism.
In particular, we have canonical isomorphisms 
\begin{enumerate}
\item $C_{R}(A,B) \cong \Hom_{R[x]}(T_{R}(A),T_{R}(B))$ of $R$-modules;
\item $C_{R}(A) \cong \End_{R[x]}(T_{R}(A))$ of $R$-algebras.
\end{enumerate}
\end{lemma}

\begin{proof}
The function $\psi_{A,B,C}$ is an $R[x]$-module homomorphism 
if and only if $C(Av) = B(Cv)$ for all $v \in R^{n}$, which in turn is equivalent to $C \in C_{R}(A,B)$.
This gives the first claim; the remaining claims now follow easily.
\end{proof}

The following result is well known and is an easy consequence of Lemma \ref{lem:homend}.

\begin{lemma}\label{lem:similar-matrices}
Let $A,B,C \in \Mat_{n}(R)$. 
Then the following are equivalent:
\begin{enumerate}
\item $C \in C_{R}(A,B) \cap \GL_{n}(R)$,
\item $C \in \GL_{n}(R)$ and $B=CAC^{-1}$,
\item $\psi_{A,B,C}$ is an $R[x]$-module isomorphism.
\end{enumerate}
In particular, $A$ and $B$ are similar over $R$ if and only if $T_R(A) \cong T_R(B)$ as $R[x]$-modules.  
\end{lemma}

\subsection{Jacobson radicals of certain endomorphism algebras}\label{subsec:jacobson-comp}
Let $F$ be a field.
We now explicitly compute the Jacobson radical $\jac(\End_{F[x]}(V))$
of $\End_{F[x]}(V)$ for a finitely generated $F[x]$-module $V$.
The motivating application is Proposition \ref{prop:wedderburnpoly} below.

\begin{lemma}\label{lem:image-of-comps-of-homs}
Let $f \in F[x]$ be an irreducible polynomial and let $j,k \in \Z_{>0}$. 
Let
\[
\lambda \in \Hom_{F[x]}(F[x]/(f^{j}), F[x]/(f^{k}))
\quad \text{ and } \quad
\mu \in \Hom_{F[x]}(F[x]/(f^{k}), F[x]/(f^{j})).
\]
\begin{enumerate}
\item If $j \leq k$ then $\im(\lambda) \sseq f^{k-j}\cdot \left( F[x]/(f^k) \right)$.
\item For any choice of $j,k$ we have $\im(\lambda\circ\mu) \sseq f^{|k-j|} \cdot \left( F[x]/(f^k) \right)$.
\end{enumerate}
\end{lemma}

\begin{proof}
Suppose $j \leq k$. We have $\lambda(x+(f^{j})) = y + (f^{k})$ for some $y \in F[x]$. 
Then 
\[
f^{j}y + (f^{k}) 
= f^{j}(y + (f^{k}))
= f^{j}\lambda(x+f^{j})
= \lambda(f^{j}x+(f^{j}))
= \lambda(0)=0,
\]
so $f^{j}y \in (f^{k})$ and hence $y \in (f^{k-j})$. 
Thus (a) follows from the fact that the image of $\lambda$ is uniquely determined by the image of $x+(f^j)$.
Part (b) follows easily from (a).
\end{proof}

\begin{prop}\label{prop:Jac-rad-Wed}
Let $f \in F[x]$ be an irreducible polynomial.
Let $m \in \Z_{>0}$, let $d_{1}, \ldots, d_{m} \in \Z_{\geq 0}$ and let
$V = \bigoplus_{j = 1}^{m} (F[x]/(f^j))^{d_{j}}$.
Then we have a canonical isomorphism
\begin{equation}\label{eq:end-decomp}
\End_{F[x]}(V) \cong E := \bigoplus_{j=1}^{m} \bigoplus_{k=1}^{m} e_{jk} \Hom_{F[x]}((F[x]/(f^k))^{d_k}, (F[x]/(f^j))^{d_j}),
\end{equation}
where the right hand side denotes the $m \times m$ `matrix ring' with $(j,k)$-th entries in $\Hom_{F[x]}((F[x]/(f^k))^{d_k}, (F[x]/(f^j))^{d_j})$.
For $1 \leq j,k \leq m$, define $\gamma_{jk}=f$ if $j=k$ and $\gamma_{jk}=1$ otherwise.
  Then the isomorphism $\End_{F[x]}(V) \cong E$ induces isomorphisms
\begin{gather}
\jac(\End_{F[x]}(V)) \cong I := \bigoplus_{j=1}^{m} \bigoplus_{k=1}^{m} e_{jk} \gamma_{jk} \Hom_{F[x]}((F[x]/(f^k))^{d_k}, (F[x]/(f^j))^{d_j}), \text{ and } \label{eq:end-jac-decomp} \\
\End_{F[x]}(V)/{\jac(\End_{F[x]}(V))} \cong E/I \cong 
\prod_{j = 1}^{m} \Mat_{d_j}(F[x]/(f)). \label{eq:end-jac-quotient}  
\end{gather}
\end{prop}

\begin{proof}
The decomposition \eqref{eq:end-decomp} follows from standard properties of Homs and direct sums.
It follows from Lemma \ref{lem:image-of-comps-of-homs} (b) 
that $I$ is a two-sided ideal of $E$.
Moreover, it is straightforward to check that $E/I$ is canonically isomorphic to the right hand side of
\eqref{eq:end-jac-quotient}.
Thus $E/I$ is artinian semisimple and so $\jac(E/I) =0$ by \cite[(5.18)]{curtisandreiner_vol1}.
Hence $\jac(E) \subseteq I$ by \cite[(5.6)(ii)]{curtisandreiner_vol1}.
  Lemma \ref{lem:image-of-comps-of-homs} (a) and the definition of $\gamma_{jk}$ implies that each element $\lambda = (\lambda_{jk}) \in I$ is an upper triangular
matrix in the sense that for $j \ge k$ the image of $\lambda_{jk}$ is contained in $f \cdot F[x]/(f^{j})$.
Hence for
$\mu = (\mu_{jk}) \in I^{m}$ the image of each $\mu_{jk}$ is contained in  $f \cdot F[x]/(f^{j})$. 
It follows that $I^{m^2} = 0$ and thus $I$ is nilpotent.
Thus, since $E$ is artinian, $I \subseteq \jac(E)$ by \cite[(5.15)]{curtisandreiner_vol1}.
Therefore $\jac(E)=I$, as claimed.
\end{proof}

\begin{corollary}\label{cor:wedderburn}
Let $r \in \Z_{>0}$ and let $V = \bigoplus_{i=1}^{r} V_{i}$ where
$V_{i} = \bigoplus_{j = 1}^{m_{i}} (F[x]/(f_{i}^{j}))^{d_{i,j}}$ for some 
$m_{i} \in \Z_{>0}$, $d_{i,j} \in \Z_{\geq 0}$, and some distinct monic 
irreducible polynomials $f_{i} \in F[x]$.
Then there are canonical isomorphisms
\[
\End_{F[x]}(V) \cong \prod_{i=1}^{r} \End_{F[x]}(V_{i}) \cong 
\prod_{i=1}^{r} \bigoplus_{j=1}^{m_{i}} \bigoplus_{k=1}^{m_{i}} 
e_{jk} \Hom_{F[x]}((F[x]/(f_{i}^{k}))^{d_{i,k}}, (F[x]/(f_{i}^{j}))^{d_{i,j}}) 
\]
and 
\[
\End_{F[x]}(V)/{\jac(\End_{F[x]}(V))}
\cong
\prod_{i=1}^{r} \prod_{j = 1}^{m_{i}} \Mat_{d_{i, j}}(F[x]/(f_{i})). 
\] 
In particular, if $F$ is a number field then $\End_{F[x]}(V)$ satisfies hypothesis \textup{(H)}.
\end{corollary}

\begin{proof}
The desired result follows from Proposition \ref{prop:Jac-rad-Wed} together with the observation that
$\Hom_{F[x]}(V_i, V_j) = 0$ for $i \neq j$.
\end{proof}

\begin{prop}\label{prop:semisimple-etale-nilpotent}
Let $n \in \Z_{>0}$ and let $A \in \Mat_{n}(F)$.
\begin{enumerate}
\item The minimal polynomial of $A$ is squarefree if and only if $C_{F}(A)$ is semisimple.
\item The minimal polynomial of $A$ is equal to the characteristic polynomial of $A$ if and only if 
$C_{F}(A)/{\jac(C_{F}(A))}$ is isomorphic to a finite product of fields.
\item The characteristic polynomial of $A$ is squarefree if and only if $C_{F}(A)$ is isomorphic to a finite product of fields.
\item If $A$ is nilpotent then $C_{F}(A)/{\jac(C_{F}(A))}$ is isomorphic to 
$\prod_{j = 1}^{m} \Mat_{d_{j}}(F)$ for some $m, d_{1}, \ldots, d_{m} \in \Z_{>0}$.
\end{enumerate}
\end{prop}

\begin{proof}
Let $f \in F[x]$ denote the characteristic polynomial of $A$.
It is a standard result in linear algebra that there is an isomorphism of $F[x]$-modules
\begin{equation}\label{eq:invariant-factor-decomp}
T_{F}(A) \cong F[x]/(g_{1}) \oplus \cdots \oplus F[x]/(g_{s}),  
\end{equation}
where $g_{1}, \ldots, g_{s} \in F[x]$ are the invariant factors of $A$
and $g_{1} \mid g_{2} \mid \cdots \mid g_{s}$ . 
Thus $g_{s}$ is the minimal polynomial of $A$ and $f=g_{1}\cdots g_{s}$.
Moreover, by Lemma~\ref{lem:homend} (b) there is a canonical isomorphism  
$C_{F}(A) \cong \End_{F[x]}(V)$ of $F$-algebras where $V:=T_{F}(A)$.
Let $f_{1}, \ldots, f_{r} \in F[x]$ denote the distinct monic irreducible factors of $f$.
Then there exists a decomposition $V= \bigoplus_{i=1}^{r} V_{i}$ and isomorphisms
$V_{i} \cong \bigoplus_{j = 1}^{m_i} F[x]/(f_{i}^{j})^{d_{i,j}}$ 
for some $m_{i} \in \Z_{>0}$ and $d_{i,j} \in \Z_{\geq 0}$. 
(a) Observe that $g_{s}$ is squarefree if and only if each
$g_{k}$ is squarefree if and only if $m_{i}=1$ for $i=1,\ldots,r$.
By Corollary \ref{cor:wedderburn}, this in turn is equivalent to 
the triviality of $\jac(\End_{F[x]}(V))$, which is equivalent to the 
semisimplicity of $\End_{F[x]}(V)$ by \cite[(5.18)]{curtisandreiner_vol1}.
(b) Observe that $g_{s}=f$ if and only if $s=1$ if and only if $d_{i,1}=1$ for $i=1,\ldots,r$. 
By Corollary \ref{cor:wedderburn}, this in turn holds if and only if 
$\End_{F[x]}(V)/{\jac(\End_{F[x]}(V))}$
is isomorphic to a finite product of fields.
(c) This follows from the previous two parts, once one obverses that 
if  $f$ is squarefree then it must be equal to $g_{s}$.
(d) If $A$ is nilpotent then $f$ is some power of $x$, 
and so $r=1$ and $f_{1}=x$. 
Thus the claim follows from Proposition~\ref{prop:Jac-rad-Wed} 
and the canonical isomorphism $F[x]/(x) \cong F$.
\end{proof}

\begin{prop}\label{prop:wedderburnpoly}
Let $K$ be a number field, let $n \in \Z_{>0}$ and let $A \in \Mat_{n}(K)$.
Let $f$ be the characteristic polynomial of $A$ and let
$f=f_{1}^{n_{1}} \cdots f_{r}^{n_{r}}$ be its factorisation, 
where $f_{1}, \ldots, f_{r} \in K[x]$ are distinct monic irreducible polynomials and $n_{i} \in \Z_{>0}$ for each $i$.
Let $K_{i}=K[x]/(f_{i})$ for $i=1,\ldots,r$.
Then there exists a polynomial-time algorithm that computes the factorisation of $f$, 
computes $C_{K}(A)$ and $\jac(C_K(A))$, and computes an explicit homomorphism of $K$-algebras 
\[
\rho \colon C_K(A) \longrightarrow C_K(A)/{\jac(C_K(A))}
\stackrel{\cong}{\longrightarrow} 
\prod_{i=1}^{r}  \prod_{j = 1}^{m_{i}} \Mat_{d_{i,j}}(K_i),
\]
for some $m_{i} \in \Z_{>0}$ and $d_{i,j} \in \Z_{\geq 0}$ such that $\sum_{j=1}^{m_{i}} jd_{i,j} = n_{i}$ for each $i$.
In particular, this solves  $\pwedderburn$ for $C_K(A)/{\jac(C_K(A))}$, which satisfies hypothesis \textup{(H)}.
\end{prop}

\begin{proof}
Assume the setup and notation of the proof of Proposition \ref{prop:semisimple-etale-nilpotent} with $F=K$.
The isomorphism of \eqref{eq:invariant-factor-decomp}
is obtained when computing the rational canonical form, which can be performed 
in polynomial time (see \cite[Theorem 4]{Villard1993}). 
  Moreover, polynomials in $K[x]$ can be factored in polynomial time by the algorithm of \cite[(4.5) Theorem]{MR774816}.
Thus we can explicitly compute a decomposition $T_{K}(A) = \bigoplus_{i=1}^{r} V_{i}$ and isomorphisms
$V_{i} \cong \bigoplus_{j = 1}^{m_i} (K[x]/(f_{i}^{j}))^{d_{i,j}}$ for some $m_{i} \in \Z_{>0}$ and
$d_{i,j} \in \Z_{\geq 0}$. 
Note that for each $i$, we have $\sum_{j=1}^{m_{i}} jd_{i,j} = n_{i}$ since $f=g_{1}\cdots g_{s}$.
Since $C_K(A)$ is canonically isomorphic to $\End_{K[x]}(T_{K}(A))$ by Lemma~\ref{lem:homend} (b),
the desired result now follows from Corollary~\ref{cor:wedderburn}.
\end{proof}

\subsection{An algorithm for determining similarity and computing a conjugating matrix}\label{subsec:alg-similarity}
We now consider the following problem.

\begin{problem*}[\psimilar] 
Given a number field $K$ with ring of integers $\mathcal{O}=\mathcal{O}_{K}$,
an integer $n \in \Z_{>0}$ and two matrices 
$A,B \in \Mat_{n}(\mathcal{O})$, determine whether $A$ and $B$ are similar over $\mathcal{O}$,
and if so, return a conjugating matrix $C \in \GL_{n}(\mathcal{O})$ such that $B= CAC^{-1}$.
\end{problem*}

Let $n \in \Z_{>0}$ and let $A,B \in \Mat_{n}(\Z)$. 
Assume that there exists $D \in \GL_{n}(\Q)$ such that $B=DAD^{-1}$.
Thus $A$ and $B$ have the same minimal polynomial $f \in \Z[x]$, 
and so the $\Z[x]$-modules $T_{\Z}(A)$ and $T_{\Z}(B)$ are in fact $\Z[x]/(f)$-lattices. 
In view of Lemma \ref{lem:similar-matrices}, this implies that 
$\psimilar$ over $\Z$ can be reduced to the problem of determining whether the $\Z[x]/(f)$-lattices
$T_{\Z}(A)$ and $T_{\Z}(B)$ are isomorphic, and if so, of computing an isomorphism between them.

Using this observation, Sarkisyan \cite{zbMATH03665057} and Grunewald \cite{MR579942} independently
showed that the conjugacy problem over $\Z$ for arbitrary pairs of matrices is decidable.
Moreover, Applegate and Onishi \cite{MR617778,MR656422} considered the 
cases of $2 \times 2$ and $3 \times 3$ matrices, and Behn and Van der Merwe \cite{MR1897818} also considered the $2 \times 2$ case.

In the case that the characteristic polynomial of $A$ (and $B$) is
squarefree (and thus the minimal and characteristic polynomials coincide), 
the above approach via $\Z[x]/(f)$-lattices is equivalent to a classical result of 
Latimer--MacDuffee~\cite{MR1503108}.
This last result was recently generalised in the dissertation of Husert \cite{Hus17}
to the case where the minimal polynomial is squarefree but
the characteristic polynomial is arbitrary.
See Proposition \ref{prop:semisimple-etale-nilpotent} for properties of 
the $\Q$-algebra $C_{\Q}(A)$ in both of these special cases.

For a discussion of practical algorithms that have been implemented on a computer,
see \S \ref{subec:similarity-implementation}.

\begin{prop}\label{prop:similarity-gen-unit-for-apps}
Let $R$ be a noetherian integral domain with field of fractions $K \neq R$.
Let $n \in \Z_{>0}$ and let $A,B \in \Mat_{n}(R)$.
Suppose that $D \in \GL_{n}(K)$ satisfies $B=DAD^{-1}$.
Then $A$ and $B$ are similar over $R$ if and only if 
\begin{enumerate}
\item the $C_R(B)$-lattice $C_R(A, B)D^{-1}$ in $C_{K}(B)$ is free of rank $1$, and
\item every (any) free generator $C'$ of $C_R(A, B)D^{-1}$ over $C_{R}(B)$ satisfies $C'D \in \GL_n(R)$.
\end{enumerate}
Furthermore, when this is the case, $B=CAC^{-1}$ where $C:=C'D$.
\end{prop}

\begin{proof}
By Lemma \ref{lem:similar-matrices-gives-isomorphism} the map 
$\theta_{D^{-1}} : C_{K}(A,B) \rightarrow C_{K}(B)$, $X \mapsto XD^{-1}$
is an isomorphism of $C_{K}(B)$-modules.
Hence the desired result follows from Proposition \ref{prop:similarity-gen-unit}. 
\end{proof}

The main algorithm of this section is as follows.

\begin{algorithm}\label{alg:similar} 
Let $K$ be a number field with ring of integers $\mathcal{O}=\mathcal{O}_{K}$,
let $n \in \Z_{>0}$, and let $A, B \in \Mat_{n}(\mathcal{O})$.
The following steps solve $\psimilar$ for $A$ and $B$, that is, they 
determine whether $A$ and $B$ are similar over $\mathcal{O}$,
and if so, return an element $C \in \GL_{n}(\mathcal{O})$ such that $B = CAC^{-1}$.
\renewcommand{\labelenumi}{(\arabic{enumi})}
\begin{enumerate}
\item Check whether $A$ and $B$ are similar over $K$, and if so, 
compute $D \in \GL_{n}(K)$ such that $B = DAD^{-1}$. 
If not, then $A$ and $B$ are not similar over $\mathcal{O}$.
\item 
Compute $C_{K}(B)$, $\jac(C_K(B))$, and an explicit homomorphism of $K$-algebras 
\[
  \rho \colon C_K(B) \longrightarrow C_K(B)/{\jac(C_K(B))} 
\stackrel{\cong}{\longrightarrow} 
\prod_{i=1}^{t} \Mat_{d_i}(K_i),
\]
where the $K_{i}$'s are (not necessarily distinct) finite field extensions of $K$.
\item Check whether $\rho(C_\calO(A, B)D^{-1})$ is a free $\rho(C_\calO(B))$-lattice,
and if so, compute a generator $E \in \rho(C_\calO(A, B)D^{-1})$.
If not, then $A$ and $B$ are not similar over $\mathcal{O}$.
\item Compute $C' \in C_\calO(A, B)D^{-1}$ such that $\rho(C')=E$.
\item Check whether $C := C'D \in \GL_n(\calO)$.
If so, then $B=CAC^{-1}$. If not, then $A$ and $B$ are not similar over $\mathcal{O}$.
\end{enumerate}
\end{algorithm}

\begin{proof}[Proof of correctness of Algorithm \ref{alg:similar}]
If all steps succeed, then $C \in C_{\mathcal{O}}(A, B) \cap \GL_{n}(\mathcal{O})$ 
and it easily follows that $B=CAC^{-1}$.
It remains to show that if any of Steps (1), (3) or (5) fail, then $A$ and $B$ are not similar over 
$\mathcal{O}$. If Step (1) fails, then this is clear.
If Step (3) fails, then Theorem~\ref{thm:lift-gen-from-semisimple-case} (a) implies that 
$C_\calO(A, B)D^{-1}$ is not free over $C_{\mathcal{O}}(B)$, and the result follows from
Proposition~\ref{prop:similarity-gen-unit-for-apps} (a).
Finally, suppose that Step (5) fails, that is, $C \notin \GL_{n}(\mathcal{O})$. 
If $C'$ is not a free generator of $C_{\mathcal{O}}(A, B)D^{-1}$ over $C_{\mathcal{O}}(B)$,
then Theorem~\ref{thm:lift-gen-from-semisimple-case} (b) implies that 
$C_{\mathcal{O}}(A, B)D^{-1}$ is not free over $C_{\mathcal{O}}(B)$, 
and again the result follows from Proposition~\ref{prop:similarity-gen-unit-for-apps} (a). 
If $C'$ is a free generator of $C_{\mathcal{O}}(A, B)D^{-1}$ over $C_{\mathcal{O}}(B)$,
then the result follows from Proposition~\ref{prop:similarity-gen-unit-for-apps} (b).
\end{proof}

The following result analyses the complexity of Algorithm \ref{alg:similar}, and further details
on each step are given in the proof.

\begin{theorem}\label{thm:isconjugated}
Let $K$ be a number field with ring of integers $\mathcal{O}=\mathcal{O}_{K}$,
let $n \in \Z_{>0}$, and let $A, B \in \Mat_{n}(\mathcal{O})$.
Let $f_{1}, \ldots, f_{r} \in K[x]$ be the distinct monic irreducible factors of the characteristic polynomial
of $B$. For $i=1,\ldots,r$ let $K_{i}=K[x]/(f_{i})$.
Let $\Lambda$ be the image of $C_{\calO}(B)$ under the projection $C_K(B) \to C_K(B)/{\jac(C_K(B))}$.
Let $\mathcal{M}$ be any choice of maximal $\mathcal{O}$-order in $C_K(B)/\jac(C_K(B))$ containing $\Lambda$ and let $\mathfrak{h} = [\mathcal{M} : \Lambda]_{\mathcal{O}}$ be the module 
index of $\Lambda$ in $\mathcal{M}$.
Then Algorithm~\ref{alg:similar} reduces the problem $\psimilar$ for $A$ and $B$ in probabilistic polynomial time to
\begin{enumerate}
\item $\pfactor(\disc(\Lambda))$, the factorisation of the discriminant of $\Lambda$,
\item for each $i$ with $1 \leq i \leq r$, one instance of $\ppip_{\mathcal O_{K_i}}$,
\item for each $i$ with $1 \leq i \leq r$, $\unit(\mathcal O_{K_i})$,
\item for each prime ideal divisor $\mathfrak{p}$ of $\mathfrak{h}$, the problem $\pdlog$ 
for extensions of $\mathcal{O}/\mathfrak{p}$, and
\item for each prime ideal divisor $\mathfrak{p}$ of $\mathfrak{h}$, the problem $\pprimitive$ for extensions of $\mathcal{O}/\mathfrak{p}$.
\end{enumerate}
Note that $\mathcal{M}$ and $\mathfrak{h}$ are not part of the input and $\mathfrak{h}$
is only needed for the above complexity statement. 
Moreover, $\mathfrak{h}$ does not depend on the choice of $\mathcal{M}$.
\end{theorem}

\begin{proof}
In the following, the steps refer to those of Algorithm \ref{alg:similar}.
Step (1) can be performed in polynomial time by \cite[Theorem 2]{MR1809971} and
Step (2) can be performed in polynomial time by Proposition \ref{prop:wedderburnpoly}.
Steps (4) and (5) are straightforward and can both be performed in polynomial time.
Step~(3) can be performed using Algorithm \ref{alg:main}, and so the desired result now
follows from Theorem \ref{thm:main}, after noting that $\pwedderburn(C_K(B)/{\jac(C_K(B)))}$ 
was already performed in Step (2).
\end{proof}

We also record the following two consequences of Remark~\ref{remark:comp}.

\begin{corollary}\label{cor:isconjugated1}
The problem $\psimilar$ reduces in probabilistic subexponential time to the problems \ppip{} and \punit{} for rings of integers of number fields.
\end{corollary}

\begin{corollary}\label{cor:isconjugated2}
There exists a polynomial quantum algorithm for solving $\psimilar$.
\end{corollary}

\subsection{Implementation of the algorithm}\label{subec:similarity-implementation}
The algorithm for solving the principal ideal problem for orders in algebras satisfying hypothesis 
(H) and its application to the similarity problem has been implemented using the computer algebra package~\textsc{Hecke}~\cite{Fieker2017} (also available in \textsc{Oscar}~\cite{Oscar}) and is included from version 0.13 onwards.
The implementation works for arbitrary pairs of matrices in $\Mat_{n}(\Z)$.
We now give a brief comparison with other algorithms and implementations, 
all of which are for pairs of matrices in $\Mat_{n}(\Z)$, subject to certain further restrictions in
cases (a)--(c). Recall that in Proposition \ref{prop:semisimple-etale-nilpotent} the restrictions in (b) and (c) are rephrased in terms of the algebra $C_{\Q}(A)$.
\begin{enumerate}
\item 
The algorithm of Opgenorth--Plesken--Schulz \cite{OPSc1998} solves the similarity problem 
for pairs of matrices of finite order.
\item
The algorithm of Husert~\cite{Hus17} solves the similarity problem for pairs of
matrices, both of which are either nilpotent or have squarefree minimal polynomial. 
However, the implementation is restricted to nilpotent matrices and matrices with irreducible minimal polynomial.
\item 
The algorithm of Marseglia~\cite{MR4111932} solves the similarity problem for pairs of
matrices with squarefree characteristic polynomial
(this condition implies that the minimal and characteristic polynomials coincide).
\item The algorithm of Eick--O'Brien and the second named author of the present article 
\cite{MR4048718} is based on ideas of Grunewald~\cite{MR579942} and solves the similarity problem for arbitrary pairs of matrices.
\end{enumerate}

All of the above algorithms (a)--(d) have been implemented in \textsc{Magma} \cite{Magma}, 
but no formal complexity analysis has been given for any of them.
However, we can compare these with our algorithm using timings and heuristic reasoning.
All timings in the examples below were performed using 
a single core of a 3.40GHz Intel E5-2643 processor and under the assumption of GRH. 
\textsc{Magma} V2.23-3 was used to run algorithms (a)--(d).

For random pairs of matrices of a given rational canonical form, our algorithm dramatically outperforms (a)
and the algorithm for nilpotent matrices of (b). In the latter case this is not surprising,
since the algorithm in question requires an exhaustive search among candidates within a large search space.
In the case of matrices with squarefree minimal polynomial, the bottleneck of algorithm (b)
is a final enumeration over a set $\Lambda/\frf$, which our algorithm 
avoids by means of the results of \S \ref{sec:comp-gens} 
(in particular, see Proposition \ref{prop:compk1}).
In cases where the set $\Lambda/\frf$ is large, our algorithm dramatically outperforms that of (b).

\begin{example}
Consider the two matrices
\[
A  = \left(\begin{smallmatrix}
0&1&0&0&0&0\\
-5336100&0&0&0&0&0\\
0&0&0&1&0&0\\
0&0&-5336100&0&0&0\\
0&0&0&0&0&1\\
0&0&0&0&-5336100&0
\end{smallmatrix}\right),
\quad
      B = \left(\begin{smallmatrix}
          0&1&5&0&53361000&0\\
-5336100&0&40&-5&0&-53361000\\
0&0&-8&1&0&0\\
0&0&-5336164&8&0&0\\
0&0&0&0&0&1\\
0&0&0&0&-5336100&0
\end{smallmatrix}\right),
\] 
both with irreducible minimal polynomial $f = x^2 + 5336100$ and characteristic polynomial $f^3$.
The algorithm of (b) requires an enumeration over a set of size
$2357947691 \approx 10^9$, thus rendering it impractical for this example.
However, the implementation of our algorithm requires 6 seconds to recognise that 
$A$ and $B$ are similar over $\Z$ and to find a conjugating matrix.
Note that $C_\Q(A) \cong \Mat_3(K)$, 
where $K = \Q[x]/(x^2 + 5336100)$.
\end{example}

Algorithm (c) is more restricted than (b) in that it requires the matrices in question to have squarefree characteristic polynomial.
However, in contrast to the squarefree minimal polynomial case of (b), it avoids a final enumeration step, 
and thus it performs as well as our algorithm in this special case. 

We have compared the implementation of our algorithm with that of algorithm (d)
for a variety of different examples and found that in all cases the former outperformed the latter, often dramatically.
However, we should mention that as a by-product, given a matrix $A \in \Mat_{n}(\Z)$,
algorithm~(d) can be used to determine generators of the arithmetic group 
$C_{\Z}(A)^{\times} = \{ X \in \GL_{n}(\Z) \mid XA = AX \}$.
Various examples in~\cite{MR4048718} as well as
the overall strategy of finding candidates in large search spaces
suggest that algorithm (d) has at least exponential complexity. 
We now review some of these examples from 
\cite{MR4048718} and show how our algorithm fares in comparison.

\begin{example}[{{\cite[6.3.2]{MR4048718}}}]\label{ex:different-min-and-char-polys}
  Consider the two matrices 
  \[
    A = \left( \begin{smallmatrix}
   -3 &  -1 &  3 &   0 &   0 &  0 &   0 &   0 &  0 \\
    1 &   0 &  0 &   0 &   0 &  0 &   0 &   0 &  0 \\
   -5 &   0 &  1 &   0 &   0 &  0 &   0 &   0 &  0 \\
    0 &   0 &  0 &  -3 &  -1 &  3 &   0 &   0 &  0 \\
    0 &   0 &  0 &   1 &   0 &  0 &   0 &   0 &  0 \\
    0 &   0 &  0 &  -5 &   0 &  1 &   0 &   0 &  0 \\
    0 &   0 &  0 &   0 &   0 &  0 &  -3 &  -1 &  3 \\
    0 &   0 &  0 &   0 &   0 &  0 &   1 &   0 &  0 \\
    0 &   0 &  0 &   0 &   0 &  0 &  -5 &   0 &  1
  \end{smallmatrix} \right)
, \quad 
B = \left(\begin{smallmatrix}
    13 &  -15 &   16 &    24 &   -16 &   -7 &   -35 &    15 &    0 \\
    -3 &   44 &  -40 &   -71 &    62 &   28 &   157 &   -76 &   16 \\
    18 &  -15 &   -3 &    -7 &   -31 &    6 &  -226 &   129 &  -52 \\
   -69 &   72 &  -55 &   -78 &    86 &   18 &   355 &  -186 &   48 \\
   -75 &   98 &  -82 &  -124 &   117 &   35 &   406 &  -206 &   46 \\
   -45 &   19 &  -21 &   -22 &    10 &    1 &    49 &   -25 &   -3 \\
    24 &  -66 &   53 &    89 &   -89 &  -31 &  -289 &   147 &  -37 \\
    30 &  -78 &   61 &   102 &  -104 &  -35 &  -348 &   178 &  -45 \\
    24 &   11 &   -8 &   -23 &    26 &   14 &    58 &   -29 &   11 
\end{smallmatrix}
\right),
\] 
both with irreducible minimal polynomial $f = x^3 + 2x^2 + 13x - 1$ and 
characteristic polynomial $f^{3}$.
As these are not equal, algorithm (c) cannot be applied in this situation.
Moreover, algorithm (d) fails to run in reasonable time 
because the search space is too large.
However, the implementation of our algorithm requires 10 seconds to recognise that 
$A$ and $B$ are similar over $\Z$ and to find a conjugating matrix.
Note that $C_{\Q}(A) \cong \Mat_3(K)$,
where $K = \Q[x]/(f)$.
\end{example}

\begin{example}[{{\cite[6.3.3]{MR4048718}}}]
Consider the two matrices
\[
A = \left(\begin{smallmatrix}
  13 &     67 &  6 &  0 &  0 &  -1 \\
   0 &      1 &  3 &  0 &  0 &   0 \\
   0 &      0 &  1 &  0 &  0 &   0 \\
-270 &  -1350 &  0 &  1 &  2 &  20 \\
-135 &   -675 &  0 &  0 &  1 &  10 \\
 -27 &   -135 &  0 &  0 &  0 &   2
\end{smallmatrix}\right), \quad
B = \left(\begin{smallmatrix}
  13 &     79 &  0 &  0 &    1 &   -76 \\
   0 &      1 &  0 &  0 &    0 &     3 \\
-270 &  -1620 &  1 &  2 &  -20 &  1620 \\
-135 &   -810 &  0 &  1 &  -10 &   810 \\
  27 &    162 &  0 &  0 &    2 &  -162 \\
   0 &      0 &  0 &  0 &    0 &     1
\end{smallmatrix} \right),
\]
both with minimal and characteristic polynomial equal to $(x - 1)^{4}(x^2 - 15x - 1)$.
As this is not squarefree, algorithms (b) and (c) cannot be applied in this situation.
Again, the search space for a certain subproblem is too large,
making the computation infeasible for algorithm (d).
However, the implementation of our algorithm finds a conjugating matrix in less than one second.
Note that $\dim_\Q(C_\Q(A)) = 6$ and 
\[ 
C_{\Q}(A)/{\jac(C_{\Q}(A))} \cong \Q \times K, 
\]
where $K = \Q[x]/(x^2 - 15x - 1)$.
\end{example}

\begin{example}
Consider the two matrices
\[ 
A = \left(\begin{smallmatrix}
      1 & -4 & 0 & 0 & 1 & 0 \\
      0 & 1 & 0 & 0 & 0 & 0 \\
      0 & 0 & 1 & -3 & -6 & 0 \\
      0 & 0 & 0 & 1 & 2 & 0 \\
      -4 & 16 & -3 & 0 & -5 & -6 \\
      0 & 0 & -37 & 0 & -9 & -55
\end{smallmatrix}\right), 
\quad
B = \left(\begin{smallmatrix}
-88 & -4 & 0 & -66 & -51 & 32 \\
-2683 & 225 & 326 & -2670 & 1755 & 634 \\
-2607 & -666 & -332 & -525 & -6747 & 2835 \\
14 & 0 & 0 & 13 & 2 & 0 \\
523 & 38 & -3 & 330 & 440 & -325 \\
285 & 74 & 37 & 54 & 749 & -314
\end{smallmatrix}\right), \quad
\] 
both with minimal and characteristic polynomial $(x - 1)^2 (x^4 + 58x^3 + 88x^2 + 176x + 1)$.
As this is not squarefree, algorithms (b) and (c) cannot be applied in this situation.
Moreover, the implementation of algorithm (d) requires approximately 1 hour to find a conjugating matrix.
By contrast, the implementation of our algorithm finds such a matrix in less than one second.
Note that $\dim_{\Q}(C_{\Q}(A)) = 6$ and 
\[ 
C_{\Q}(A)/{\jac(C_{\Q}(A))} \cong \Q \times K, 
\]
where $K = \Q[x]/(x^4 + 58x^3 + 88x^2 + 176x + 1)$.
\end{example}

\section{Application: Galois module structure of rings of integers}\label{sec:GMS}

An important motivation for  Algorithm~\ref{alg:main} and its predecessors
is the investigation of the Galois module structure of rings of integers. 
We only briefly recall the problem here and refer the reader to the introduction
of \cite{MR4136552} for a more detailed overview.

Let $L/K$ be a finite Galois extension of number fields and let $G=\Gal(L/K)$.
The classical Normal Basis Theorem says that $L \cong K[G]$ as $K[G]$-modules.
A much more difficult problem is that of determining whether the ring of integers
$\mathcal{O}_{L}$ is free over its so-called associated order 
$\mathcal{A}_{L/K} = \{ \alpha \in K[G] \mid \alpha\mathcal{O}_{L} \subseteq \mathcal{O}_{L} \}$.
Note that if a prime $\mathfrak{p}$ of $K$ is (at most) tamely ramified in $L/K$
or is such that the localised associated order $\mathcal{A}_{L/K, \mathfrak{p}}$ is maximal, 
then the localisation $\mathcal{O}_{L,\mathfrak{p}}$ is necessarily free over $\mathcal{A}_{L/K, \mathfrak{p}}$. In particular, $\mathcal{O}_{L}$ is locally free over $\mathcal{A}_{L/K}$ if and only if
$\mathcal{O}_{L,\mathfrak{p}}$ is free over $\mathcal{A}_{L/K, \mathfrak{p}}$ for every prime $\mathfrak{p}$
of $K$ that is wildly ramified in $L/K$.
In this situation, one can consider the class
$[\mathcal{O}_{L}]$ in the locally free class group $\Cl(\mathcal{A}_{L/K})$. 
Moreover, if $K[G]$ satisfies hypothesis (H) then every order in $K[G]$
has the so-called locally free cancellation property 
(this follows from Jacobinski's cancellation theorem~\cite[(51.24)]{curtisandreiner_vol2}),
and so $\mathcal{O}_{L}$ is free over $\mathcal{A}_{L/K}$ if and only if it is locally free
and the class $[\mathcal{O}_{L}]$ is the trivial element of $\Cl(\mathcal{A}_{L/K})$.

For an abstract finite group $\Gamma$, we say that $L/K$ is a $\Gamma$-extension if it is a Galois extension such that 
$\Gal(L/K) \cong \Gamma$. 
Let $\Gamma=S_{4} \times C_{2}$, the direct product of the symmetric group on $4$
letters and the cyclic group of order $2$. 
Since
\[ 
\Q[\Gamma] \cong \prod_{i=1}^{4} \Q \times \prod_{j=1}^{2} \Mat_2(\Q) \times 
\prod_{k=1}^{4} \Mat_3(\Q), 
\] 
the algebra $\Q[\Gamma]$ satisfies hypothesis (H).
Using the methods of~\cite{FHS2019} we have constructed wildly ramified $\Gamma$-extensions of $\Q$ 
of small discriminant.
The wildly ramified $\Gamma$-extension of minimal discriminant is  $L_1 := K_1(\sqrt{92})$, 
where $K_1$ is the $S_4$-extension of $\Q$ defined by
\begin{multline*}
x^{24} + 2 x^{22} + 27 x^{20} + 112 x^{18} + 585 x^{16} + 338 x^{14} + 5767 x^{12}  \\
+ 4362 x^{10} + 1417 x^{8} - 76 x^{6} - 29 x^{4} - 6 x^{2} + 1 \in \Q[x].  
\end{multline*}
The field $L_1$ has discriminant $2^{84} \cdot 23^{24}$ and is wildly ramified at $2$. 
Moreover, the associated order $\calA_{L_1/\Q}$ has index $2^{43} \cdot 3^{3}$ in a maximal order $\mathcal{M}$ satisfying $\calA_{L_1/\Q} \subseteq \mathcal{M} \subseteq \Q[\Gal(L_{1}/\Q)]$ (note that this index is independent of the choice of $\mathcal{M}$).
Using Algorithm~\ref{alg:main} we have checked that $\calO_{L_1}$ is free over $\calA_{L_1/\Q}$
and have also obtained an explicit generator (unfortunately, the coefficients are too large to reproduce in print).
The algorithms of~\cite{bley-boltje, BW09} show that $\Cl(\calA_{L_1/\Q}) \cong C_{2}$.
However, the algorithm of \cite{BW09} for solving the discrete logarithm problem 
in a locally free class group is restricted to the case in which the order in question 
is a group ring or a maximal order,
and so this approach does not allow us to determine $[\calO_{L_1}]$ in $\Cl(\calA_{L_1/\Q})$.

We have performed the same computation using Algorithm~\ref{alg:main} described above for all wildly ramified $\Gamma$-extensions $L/\Q$ with $\lvert \disc(L) \rvert \leq 60^{48}$. For $686$ out of these $2600$ extensions, 
$\mathcal{O}_{L}$ is locally free over $\mathcal{A}_{L/\Q}$, 
and in all of these cases, $\mathcal{O}_{L}$ is in fact free over $\mathcal{A}_{L/\Q}$.
It would be interesting to find a proof of, or counterexample to, the assertion that the same phenomenon holds without the restriction on $\lvert \disc(L) \rvert$.

\appendix

\section{Weak approximation in probabilistic polynomial time}

Let $K$ be a number field with ring of integers $\calO = \OK$.
Let $\mathfrak{a}$ and $\mathfrak{b}$ be nonzero integral ideals of $\mathcal{O}$.
A classical result (see~\cite[Corollary 1.3.9]{MR1728313}) asserts that there exists
a deterministic algorithm for computing $x \in K^{\times}$ such that $x \mathfrak{a}$ is integral 
and coprime to $\mathfrak{b}$.
If the factorisation of $\mathfrak{b}$, or equivalently, 
of $\Nm(\mathfrak{b})$, is given, the algorithm runs in polynomial time.
There also exists a probabilistic algorithm~\cite[Algorithm 1.3.14]{MR1728313}, 
which does not require the factorisation of $\mathfrak{b}$ or $\Nm(\mathfrak{b})$, but is not polynomial time.
The aim of this section is to combine the deterministic and probabilistic variants to obtain a probabilistic polynomial-time algorithm.
The approach is based on the following general form of the constructive weak approximation theorem,
which relies on ideas of~\cite[Algorithm 6.15]{Belabas2004}.
For a nonzero prime ideal $\mathfrak{p}$ of $\mathcal{O}$, let  $v_{\mathfrak{p}}(-)$ denote the
$\mathfrak{p}$-adic valuation.

\begin{prop}\label{prop:randapprox}
There exists a probabilistic polynomial-time algorithm
that given nonzero integral ideals $\mathfrak{a}$ and $\mathfrak{b}$ of $\calO$ returns an element
$x \in \mathfrak{a}$ with $v_\mathfrak{p}(x) = v_\mathfrak{p}(\mathfrak{a})$ for all prime ideals 
$\mathfrak{p}$ dividing $\mathfrak{b}$.
\end{prop}

\begin{proof}
We adapt the proofs of \cite[Lemmas 6.14, 6.16]{Belabas2004}, taking into account \cite[Remark 6.17~(2)]{Belabas2004}.
For the rest of the proof, we fix a positive constant $0<C<1$.
Let $a = \min(\fra \cap \Z_{>0})$, let $b = \min(\frb \cap \Z_{>0})$ and let $d=[K:\Q]$.
Note that if $a=1$ or $b=1$ or $d=1$ then we can just take $x=a$.
Thus we can and do assume that $a, b,d \geq 2$.
We define $y \in \R$ by the equality $Cy\log(y) = d \log(b)$. Then $y >2$ and we observe that 
$y$ is polynomially bounded in terms of $d$ and $\log(b)$.
Hence we can determine the set 
\[
S := \{ \mathfrak p \subseteq \calO \text{ prime such that }\mathfrak p \cap \Z = (p) \text{ with a rational prime } p < y \}
\]
in polynomial time. We define ideals
\[ 
\fra_0 = \prod_{\frp \in S} \frp^{v_\frp(\fra)}, \quad\quad \frb_0 = \prod_{\frp \in S} \frp^{v_\frp(\frb)}.
\]
Then $\fra = \fra_0\fra_1$ and $\frb = \frb_0\frb_1$ with integral ideals $\fra_1, \frb_1$ such that
\[
  \fra_0 + \fra_1 = \frb_0 + \frb_1 = \calO,
\]
which can be computed in polynomial time.
We write $b = b_0 b_1$ with
\[ 
b_{0} = \prod_{p < y} p^{v_{p}(b)}. 
\]
Since the factorisations of $\mathfrak{a}_{0}$ and $\mathfrak{b}_{0}$ are known,
using the deterministic polynomial-time algorithm~\cite[Proposition 1.3.8]{MR1728313}
we can find $x_0 \in \calO$ with $x_0 \in \mathfrak a_0$ and
$v_{\mathfrak p}(x_0) = v_\mathfrak p(\mathfrak a_0)$ for all $\mathfrak p$ dividing $\mathfrak b_0$.

We now show that we can find an element 
$x_{1} \in \mathfrak{a}_{1}$ with $v_{\mathfrak{p}}(x_1) = v_{\mathfrak{p}}(\mathfrak{a}_{1})$ for all
$\mathfrak{p}$ dividing $\mathfrak{b}_{1}$ in probabilistic polynomial time.
For the rest of the proof we will refer to such elements as \textit{good} elements.
We will prove that a positive proportion (independent of $\mathfrak a$ and $\mathfrak b$) of elements of the finite abelian group
$\mathfrak a_1/\mathfrak a_1\mathfrak b_1$ are good.
For a prime ideal $\mathfrak{p}$ dividing $\mathfrak{b}_1$, 
let $A_\mathfrak{p}$ denote the set 
$\mathfrak{a}_1 \mathfrak{p}/\mathfrak{a}_1 \mathfrak{b}_1$.
Then, for a set of prime ideals $T$ dividing $\mathfrak{b}_1$, we have
\[
\Bigl\lvert \bigcap_{\mathfrak{p} \in T} A_\mathfrak{p} \Bigr\rvert = \Nm(\mathfrak b_1)/\prod_{\mathfrak p \in T}\Nm(\mathfrak{p}). 
\]
From the inclusion-exclusion principle it follows that
\[
  \Bigl\lvert \bigcup_{\frp \mid \frb_1 } A_\frp \Bigr\rvert =
  \Nm(\frb_1) \left( 1 - \prod_{\frp \mid \frb_1} \left(1 - \frac{1}{\Nm(\frp)}\right) \right) .
\]
By definition, the lift of $x \in \mathfrak{a}_1/\mathfrak{a}_1\mathfrak{b}_1$ is good if and only if
$x \not\in \bigcup_{\mathfrak{p} \mid \mathfrak{b}_1} A_\mathfrak{p}$.
Hence the probability that (the lift) of a random element of 
$\mathfrak{a}_1/\mathfrak{a}_1\mathfrak{b}_1$ 
is good is
\[
\prod_{\mathfrak p \mid \mathfrak b_1}\left( 1 - \frac{1}{\Nm(\mathfrak p)} \right).
\]
Now set $C_1 := d \log(b_1)/(y \log(y)) \leq C$.
Since there are at most $d \log_y(b_1)$ prime ideals $\frp$ dividing $\frb_1$, 
each satisfying $\Nm(\frp) \geq y$, we have
\begin{align*}
\textstyle{\prod_{\mathfrak p \mid \mathfrak b_1}(1 - 1/\Nm(\mathfrak p))} 
\geq (1 - 1/y)^{d \log_y(b_1)} 
& \geq \exp(-1/y-1/y^{2})^{d \log_y(b_1)} 
= \exp(-C_1 - C_1/y) \\
&\geq \exp(-C(1+1/y)) \geq \exp(-3 C/2).
\end{align*}
Here the second inequality follows from $1 - x \geq \exp(-x - x^2)$ for $0 \leq x \leq 1/2$.
Thus we can find a good element in probabilistic polynomial time.

Now given $x_i \in \mathfrak a_i$ with $v_\mathfrak p(x_i) = v_\mathfrak p(\mathfrak a_i)$ for all primes
$\mathfrak p$ dividing $\mathfrak b_i$, we proceed as follows.
For $i=0,1$ let $\frc_i$ be the largest divisor of $\frb_i$ which is coprime to $\fra$.
Note that each $\frc_i$ can be determined in polynomial time by using only ideal sums and ideal division.
Moreover, if $\frp$ is a prime ideal with $\frp \mid \frb_i$ and $\frp \nmid \fra$, then $\frp \mid \frc_i$.
Since $\fra_0^2\frc_0 + \fra_1^2\frc_1 = \calO$, 
we can determine elements $e_i \in \fra_i^2\frc_i$ with $e_0 + e_1 = 1$ in polynomial time.
We now prove that the element
\[
  x := e_0x_1 + e_1x_0 \in \fra
\]
satisfies $v_\frp (x) = v_\frp(\fra)$ for all $\frp$ dividing $\frb$.

Case 1: $\frp \nmid \fra$. Assume that $\frp \mid \frb_1$. Then $\frp \mid \frc_1$ and hence $e_1 \in \frp$, $e_0 \not\in \frp$. Moreover,
\begin{eqnarray*}
  v_\frp(e_0x_1) &=& v_\frp(e_0) + v_\frp(x_1) =  v_\frp(e_0) + v_\frp(\fra_1) = v_\frp(e_0) = 0, \\
  v_\frp(e_1x_0) &=& v_\frp(e_1) + v_\frp(x_0) \ge  v_\frp(e_1) > 0.
\end{eqnarray*}
Hence $v_\frp(x) = \min(v_\frp(e_0x_1),  v_\frp(e_1x_0)) = 0 = v_\frp(\fra)$.
The subcase $\frp \mid \frb_0$ is similar.

Case 2: $\frp \mid \fra$. Assume that $\frp \mid \frb_1$. Then $\frp \not\in S$ and hence $\frp \nmid \fra_0$. It follows that $\frp\mid\fra_1$,
and hence $e_0 \not\in \frp, e_1 \in \frp$. Moreover,
\begin{eqnarray*}
  v_\frp(e_0x_1) &=& v_\frp(e_0) + v_\frp(x_1) =  v_\frp(e_0) + v_\frp(\fra_1) = v_\frp(e_0) + v_\frp(\fra) = v_\frp(\fra), \\
  v_\frp(e_1x_0) &=& v_\frp(e_1) + v_\frp(x_0) \ge  v_\frp(e_1) \ge 2v_\frp(\fra_1) > v_\frp(\fra_1) = v_\frp(\fra).
\end{eqnarray*}
Hence $v_\frp(x) = \min(v_\frp(e_0x_1),  v_\frp(e_1x_0)) = v_\frp(\fra)$.
The subcase $\frp \mid \frb_0$ is similar.
\end{proof}

\begin{corollary}\label{cor:coprime}
There exists a probabilistic polynomial-time algorithm that given nonzero integral ideals 
$\mathfrak{a}$ and $\mathfrak{b}$ of $\mathcal{O}$ 
returns an element $x \in K^{\times}$ such that $x \mathfrak{a}$ is integral and coprime to $\mathfrak{b}$.
\end{corollary}

\begin{proof}
We need to find an element $x \in \mathfrak{a}^{-1}$ such that 
$v_\mathfrak{p}(x) = v_\mathfrak{p}(\mathfrak{a}^{-1})$ for all prime ideals $\mathfrak{p}$ dividing
$\mathfrak{b}$.
Setting $a = \min(\mathfrak a \cap \Z_{>0})$ to be the minimum of $\mathfrak a$, this is equivalent
to $v_\mathfrak p(ax) = v_\mathfrak p(a\mathfrak a^{-1})$ for all $\mathfrak p$ dividing $\mathfrak b$.
As $a \mathfrak a^{-1}$ is integral, the result follows from Proposition~\ref{prop:randapprox} applied to $a\mathfrak{a}^{-1}$ and $\mathfrak{b}$.
\end{proof}

\begin{corollary}\label{cor:steinitz}
There exists a probabilistic polynomial-time algorithm that given
a generating set of an $\mathcal{O}$-lattice $M \subseteq K^{n}$ of rank $n$,
determines a Steinitz form of $M$, that is, elements $w_1,\dotsc,w_n \in K^{n}$
and a fractional ideal $\mathfrak{a}$ of $\mathcal{O}$ such that
\[
M = \mathcal{O} w_{1} \oplus \dotsb \oplus \mathcal{O} w_{n-1} \oplus \mathfrak{a} w_n. 
\] 
\end{corollary}

\begin{proof}
  A pseudo-Hermite normal form can be determined in probabilistic polynomial time by~\cite[Theorem 34]{Biasse2016}.
The reduction to the Steinitz form is described in~\cite[Lemma 1.2.20]{MR1728313}
and requires the computation of coprime representatives of ideal classes.
Thus the claim follows from Corollary~\ref{cor:coprime}.
\end{proof}

\bibliography{OmittingTheEnumerationStep}
\bibliographystyle{amsalpha}

\end{document}